\documentclass[11pt,twoside,a4paper]{article}
\usepackage{color,amscd,amsmath,amssymb,amsthm,latexsym,stmaryrd,authblk,mathabx,shuffle,mathrsfs,hyperref,tikz,
tkz-tab} 
\usepackage[latin1]{inputenc}
\usepackage[T1]{fontenc}   
\usepackage[english]{babel}
\providecommand{\keywords}[1]{\textbf{\textit{Keywords.}} #1}
\providecommand{\AMSclass}[1]{\textbf{\textit{AMS classification.}} #1}

\input{xy}
\xyoption{all}

\title{Contractions and extractions on twisted bialgebras and coloured Fock functors}
\date{}
\author{Lo\"ic Foissy}
\affil{\small{Univ. Littoral Côte d'Opale, UR 2597
LMPA, Laboratoire de Mathématiques Pures et Appliquées Joseph Liouville
F-62100 Calais, France}.\\ Email: \texttt{foissy@univ-littoral.fr}}

\newcommand{\tdun}[1]{\begin{picture}(10,5)(-2,-1)
\put(0,0){\circle*{2}}
\put(3,-2){\tiny #1}
\end{picture}}

\newcommand{\tddeux}[2]{\begin{picture}(12,5)(0,-1)
\put(3,0){\circle*{2}}
\put(3,5){\circle*{2}}
\put(3,0){\line(0,1){5}}
\put(6,-3){\tiny #1}
\put(6,5){\tiny #2}
\end{picture}}

\newcommand{\gdtroisun}[3]{\begin{picture}(22,12)(-8,-1)
\put(3,0){\circle*{2}}
\put(6.5,7){\circle*{2}}
\put(-1,7){\circle*{2}}
\put(-1,7){\line(1,0){7.5}}
\put(-2.8,0){\Large $\vee$}
\put(5,-2){\tiny #1}
\put(9,5){\tiny #2}
\put(-8,5){\tiny #3}
\end{picture}}

\newcommand{\tdtroisdeux}[3]{\begin{picture}(12,15)(-2,-1)
\put(0,0){\circle*{2}}
\put(0,5){\circle*{2}}
\put(0,10){\circle*{2}}
\put(0,0){\line(0,1){5}}
\put(0,5){\line(0,1){5}}
\put(3,-4){\tiny #1}
\put(3,4){\tiny #2}
\put(3,12){\tiny #3}
\end{picture}}

\theoremstyle{plain}
\newtheorem{theo}{Theorem}[section]
\newtheorem{lemma}[theo]{Lemma}
\newtheorem{cor}[theo]{Corollary}
\newtheorem{prop}[theo]{Proposition}
\newtheorem{defi}[theo]{Definition}

\theoremstyle{remark}
\newtheorem{remark}{Remark}[section]
\newtheorem{notation}{Notations}[section]
\newtheorem{example}{Example}[section]

\newcommand{\K}{\mathbb{K}}
\newcommand{\N}{\mathbb{N}}
\newcommand{\calG}{\mathscr{G}}
\newcommand{\bfG}{\mathbf{G}}

\newcommand{\id}{\mathrm{Id}}
\newcommand{\com}{\mathbf{Com}}
\newcommand{\bfP}{\mathbf{P}}
\newcommand{\bfQ}{\mathbf{Q}}
\newcommand{\bfR}{\mathbf{R}}
\newcommand{\eq}{\mathcal{E}}
\newcommand{\cl}{\mathrm{cl}}
\newcommand{\bfI}{\mathbf{1}}
\newcommand{\sym}{\mathfrak{S}}

\newcommand{\vect}{\mathrm{Vect}}

\newcommand{\comp}{\mathbf{Comp}}
\newcommand{\calcomp}{\mathscr{C}\hspace{-1mm}\mathit{omp}}
\newcommand{\QSh}{\mathrm{QSh}}
\newcommand{\QSym}{\mathbf{QSym}}

\newcommand{\bfT}{\mathbf{T}}

\newcommand{\calF}{\mathcal{F}}

\newcommand{\coinv}{\mathrm{coInv}}

\begin{document}

\maketitle

\begin{abstract}
We introduce a notion of extraction-contraction coproduct on twisted bialgebras, that is to say bialgebras in the category
of linear species. If $\bfP$ is a twisted bialgebra, a contraction-extraction coproduct sends $\bfP[X]$ to 
$\bfP[X/\sim]\otimes \bfP[X]$ for any finite set $X$ and any equivalence relation $\sim$ on $X$,
with a coassociativity constraint and compatibilities with the product and coproduct of $\bfP$. 
We prove that if $\bfP$ is a twisted bialgebra with an extraction-contraction coproduct, then $\bfP\circ \com$
is a bialgebra in the category of coalgebraic species, that is to say species in the category of coalgebras.

We then introduce a coloured version of the bosonic Fock functor. This induces a bifunctor which associates to any
bialgebra $(V,\cdot,\delta_V)$ and to any twisted bialgebra $\bfP$ with an extraction-contraction coproduct
a comodule-bialgebra $\calF_V[\bfP]$: this object inherits a product $m$ and two coproducts $\Delta$ and $\delta$,
such that $(\calF_V[\bfP],m,\Delta)$ is a bialgebra in the category of right $(\calF_V[\bfP],m,\delta)$-comodules.

As an example, this is applied to the twisted bialgebra of graphs. 
The coloured Fock functors then allow to extend the construction of the double
bialgebra of graphs to double bialgebras of graphs which vertices are decorated by elements of any bialgebra $V$. 
Other examples (on mixed graphs, hypergraphs, noncrossing partitions$\ldots$) will be given in a series of forthcoming papers.
\end{abstract}

\\

\AMSclass{16T05 16T30 18M80 05C25}

\tableofcontents

\section*{Introduction}

Bialgebras in cointeraction, or cointeracting bialgebras, or comodule-bialgebras, or double bialgebras,
 are bialgebras $(A,m,\Delta)$ with a second
coassociative and counitary coproduct $\delta$, such that $(A,m,\Delta)$ is a bialgebra in the category of 
right $(A,m,\delta)$-comodules, where the coaction on $A$ is $\delta$ itself. In particular, 
we obtain the following compatibility between the two coproducts $\Delta$ and $\delta$:
\[(\Delta\otimes \id)\circ \delta=m_{1,3,24}\circ (\delta \otimes \delta)\circ \Delta,\]
meaning that $\Delta:A\longrightarrow A\otimes A$ is a comodule morphism. The map $m_{1,3,24}:A^{\otimes 4}\longrightarrow
A^{\otimes 3}$ sends $a\otimes b\otimes c\otimes d$ onto $a\otimes c\otimes bd$. Examples are given 
by quasi-shuffle algebras \cite{Hoffman2000,Hoffman2020,Ebrahimi-Fard2017-2} (which include the polynomial algebra $\K[X]$ 
and the algebra of quasisymmetric functions $\QSym$ \cite{Aguiar2006-2,Gelfand1995,Hazewinkel2005,Malvenuto2011,Stanley1999}),
some combinatorial double bialgebras based on rooted trees \cite{Calaque2011}, various families of graphs
\cite{Manchon2012,Foissy36}, posets or finite topologies \cite{Foissy37}, noncrossing partitions \cite{Foissy38}$\ldots$
These objects play an important role in Bruned, Hairer and Zambotti's process of renormalisation of stochastic PDEs 
\cite{Bruned2015,Bruned2019}.

Twisted double bialgebras, that is to say double bialgebras in the category of species, have been studied
in \cite{Foissy39}. A (linear) species is a functor from the category of finite sets with bijections
to the category of vector spaces, and a species morphism between two species $\bfP$ and $\bfQ$ is a natural transformation
between these two functors (see Section \ref{sect1} for more details). This category of species is symmetric monoidal
with the Cauchy tensor product, allowing to define and consider algebras, coalgebras and bialgebras in the category
of species. Replacing the category of vector spaces by the category of coalgebras, we can consider coalgebraic species, 
which also form a symmetric monoidal category with the Cauchy tensor product, 
and consider algebras, coalgebras, and bialgebras in this category.
If $(\bfP,m)$ is an algebra in the category of coalgebraic species, then for any disjoint finite sets $X_1,\ldots,X_k$,
$\bfP[X_1]\otimes \ldots \otimes \bfP[X_k]$ is a $\bfP[X_1\sqcup \ldots \sqcup X_k]$-comodule, with the coaction
\[\rho_{X_1,\ldots,X_k}=m_{1,3,\ldots,2k-1,24\ldots 2k}\circ (\delta_{X_1}\otimes \ldots \otimes \delta_{X_k}),\]
where $\delta_{X_i}$ is the coproduct of the coalgebra $\bfP[X_i]$ and $m_{1,3,\ldots,2k-1,24\ldots 2k}$ is defined by
\[m_{1,3,\ldots,2k-1,24\ldots 2k}(a_1\otimes b_1\otimes \ldots \otimes a_k\otimes b_k)
=a_1\otimes \ldots \otimes a_k\otimes b_1\dots b_k.\]
In this context, a twisted double bialgebra is a bialgebra $(\bfP,m,\Delta)$ in the category
of coalgebraic species, such that for any finite sets $X$ and $Y$, $m_{X,Y}:\bfP[X]\otimes \bfP[Y]\longrightarrow \bfP[X\sqcup Y]$
and  $\Delta_{X,Y}:\bfP[X\sqcup Y]\longrightarrow\bfP[X]\otimes \bfP[Y]$ are comodule morphisms. 
An alternative way to describe these objects is to consider a second tensor products on species, namely the Hadamard tensor product.
If $(\bfP,m,\Delta,\delta)$ is a twisted double bialgebra, Aguiar and Mahajan's bosonic Fock functor $\calF$ \cite{Aguiar2010}
sends it to a double bialgebra, in the classical sense.\\

We introduce in this article a way to obtain twisted double bialgebras from extraction-contraction coproducts 
(Definition \ref{deficontraction}). These coproducts are defined on a twisted bialgebra $(\bfP,m,\Delta)$:
for any finite set $X$, for any equivalence relation $\sim$ on $X$, $\bfP[X]$ has a coproduct
\[\delta_\sim:\bfP[X]\longrightarrow \bfP[X/\sim]\otimes \bf[P],\]
satisfying some coassociativity constraints and compatibilities with the product $m$ and the coproduct $\Delta$. 
If $(\bfP,m,\Delta,\delta)$ is such an object, we prove that $\bfP\circ \com$ is a double twisted bialgebra
(Proposition \ref{propbialgebra}): here, $\circ$ is the composition of species and $\com$ is the species defined
by $\com[X]=\K$ for any finite set $X$. In other words, for any finite set $X$,
\[\bfP\circ \com[X]=\bigoplus_{\sim \in \eq[X]} \bfP[X/\sim],\]
where $\eq[X]$ is the set of equivalence relations on $X$. 

Consequently, if $(\bfP,m,\Delta,\delta)$ is a twisted bialgebra with an extraction-contraction coproduct,
$\calF[\bfP\circ \com]$ is a double bialgebra. We extend this result to a larger families of bosonic Fock functors.
For any vector space $V$, we introduce a species $\bfT_V$ of tensors, such that for any $n$,
\[\bfT_V[\{1,\ldots,n\}]=V^{\otimes n}.\]
We prove that if $(V,\delta_V)$ is a coalgebra, then $\bfT_V$ is a double twisted bialgebra (Proposition \ref{proptensor}).
Consequently, if $(V,\delta_V)$ is a bialgebra and $(\bfP,m,\Delta,\delta)$ is a twisted bialgebra with 
a contraction-extraction coproduct, then $(\bfP\circ \com) \otimes \bfT_V$ is a twisted double bialgebra and,
consequently, $\calF[(\bfP\circ \com) \otimes \bfT_V]$ is a double bialgebra. 
Moreover, if $(V,\cdot,\delta_V)$ is a commutative, not necessarily unitary bialgebra, then this double bialgebra admits
a particular quotient identified as a vector space with
\[\calF_V[\bfP]=\bigoplus_{n=0}^\infty V^{\otimes n}\otimes_{\sym_n} \bfP[\{1,\ldots,n\}],\]
where the symmetric group $\sym_n$ acts on the left of $\bfP[\{1,\ldots,n\}]$ by the species structure of $\bfP$
and on the right on $V^{\otimes n}$ by permutations of the factors of tensors. 
This defines a bifunctor $\calF_\cdot[\cdot]$, which associates to any commutative bialgebra $(V,\cdot,\delta_V)$ 
and any twisted bialgebra $(\bfP,m,\Delta)$ with an extraction-contraction coproduct $\delta$ 
a double twisted bialgebra $\calF_V[\bfP]$ (Corollary \ref{corbifoncteur}).  This bifunctor is called the 
$V$-coloured bosonic Fock functor. As an example, this bifunctor is applied to the twisted bialgebra of graphs. 
Therefore, for any commutative bialgebra $V$, we obtain a double bialgebra of graphs which vertices are decorated by elements of $V$,
with linearity relations on each vertex. The product is the disjoint union and the first coproduct $\Delta$ 
is combinatorially given by disjunction of the set of vertices into two parts. For example, if $v_1,v_2,v_3\in V$,
\begin{align*}
\Delta(\hspace{1mm}\gdtroisun{\hspace{-4mm}$v_1$}{$v_3$}{\hspace{-1mm}$v_2$}\:)&=
\hspace{1mm}\gdtroisun{\hspace{-4mm}$v_1$}{$v_3$}{\hspace{-1mm}$v_2$}\:\otimes 1
+\tddeux{$v_1$}{$v_2$}\:\otimes \tdun{$v_3$}\:
+\tddeux{$v_2$}{$v_3$}\:\otimes \tdun{$v_1$}\:+\tddeux{$v_1$}{$v_3$}\:\otimes \tdun{$v_2$}\:\\
&+1\otimes \hspace{1mm}\gdtroisun{\hspace{-4mm}$v_1$}{$v_3$}{\hspace{-1mm}$v_2$}\:
+\tdun{$v_3$}\:\otimes  \tddeux{$v_1$}{$v_2$}\:+\tdun{$v_1$}\:\otimes  \tddeux{$v_2$}{$v_3$}\:
+\tdun{$v_2$}\:\otimes \tddeux{$v_1$}{$v_3$}.
\end{align*}
This coproduct is cocommutative.
The second coproduct is given by contractions of edges (for the left side) and elimination of edges (for the right side). 
 For example, if $v_1,v_2,v_3\in V$,
\begin{align*}
\delta(\hspace{1mm}\gdtroisun{\hspace{-4mm}$v_1$}{$v_3$}{\hspace{-1mm}$v_2$}\:)&=
\hspace{1mm}\gdtroisun{\hspace{-4mm}$v_1'$}{$v_3'$}{\hspace{-1mm}$v_2'$}\:\otimes
\tdun{$v_1''$}\hspace{2mm}\tdun{$v_2''$}\hspace{2mm}
\tdun{$v_3''$}\hspace{2mm}+\tdun{$v_1'\cdot v_2'\cdot v_3'$}\hspace{10mm}\otimes \hspace{1mm}\gdtroisun{\hspace{-4mm}$v_1''$}{$v_3''$}{\hspace{-1mm}$v_2''$}\:\\
&+\tddeux{$v_1'\cdot v_2'$}{$v_3'$}\hspace{6mm} \otimes \tddeux{$v_1''$}{$v_2''$}\hspace{2mm}\tdun{$v_3''$}\hspace{2mm}
+\tddeux{$v_1'\cdot v_3'$}{$v_2'$}\hspace{6mm} \otimes \tddeux{$v_1''$}{$v_3''$}\hspace{2mm}\tdun{$v_2''$}\hspace{2mm}
+\tddeux{$v_2'\cdot v_3'$}{$v_1'$}\hspace{6mm} \otimes \tddeux{$v_2''$}{$v_3''$}\hspace{2mm}\tdun{$v_1''$}\hspace{2mm},
\end{align*}
using Sweedler's notation for the coproduct of $V$. 
this coproduct $\delta$ is not cocommutative. 
Note that, contrarily to $\Delta$, the bialgebraic structure of $V$ is needed to define $\delta$. \\

Other examples of twisted bialgebras (on mixed graphs, on hypergraphs, on partitions$\ldots$)
with an extraction-contraction coproducts will be given in a series of forthcoming papers  with various combinatorial applications.\\

This paper is organised as follows. The first section gives reminders and notations on species and twisted (bi-, co-)algebras.
It also introduces the twisted bialgebra of tensors $\bfT_V$ and the twisted double bialgebra of set compositions $\comp$. 
In the second section, extraction-contraction coproducts are introduced and studied.
The last section deals with coloured Fock functors, seen as a bifunctor giving double bialgebras.
The example of the species of graphs is detailed all along the text.\\

\textbf{Acknowledgements}. 
The author acknowledges support from the grant ANR-20-CE40-0007
\emph{Combinatoire Algébrique, Résurgence, Probabilités Libres et Opérades}.
The author thanks Pierre Catoire for his careful reading. \\

\begin{notation} \begin{enumerate}
\item We denote by $\K$ a commutative field. Any vector space in this field will be taken over $\K$.
\item For any $n\in \N$, we denote by $[n]$ the set $\{1,\ldots,n\}$. In particular, $[0]=\emptyset$.
\end{enumerate}\end{notation}

\section{Twisted algebras, bialgebras and coalgebras}

\label{sect1}

\subsection{Reminders on species and twisted objects}

Recall \cite{Joyal1981,Joyal1986} that a species is a functor $\bfP$ from the category of finite sets with bijections to the category
of vector spaces. For any finite set $X$, the vector space associated to $X$ by $\bfP$ is denoted by $\bfP[X]$.
If $\sigma:X\longrightarrow Y$ is a bijection between two finite sets, the linear bijection associated to $\sigma$
is denoted by $\bfP[\sigma]:\bfP[X]\longrightarrow \bfP[Y]$. 

If $\bfP$ and $\bfQ$ are two species, a species morphism from $\bfP$ to $\bfQ$ is a natural transformation
between these functors, that is to say, for any finite set $X$, a linear  map $f_X:\bfP[X]\longrightarrow \bfQ[X]$
such that for any bijection $\sigma:X\longrightarrow Y$ between two finite sets, the following square commutes:
\[\xymatrix{\bfP[X]\ar[r]^{f_X}\ar[d]_{\bfP[\sigma]}&\bfQ[X]\ar[d]^{\bfQ[\sigma]}\\
\bfP[Y]\ar[r]_{f_Y}&\bfQ[Y]}\]
This defines a category of species. \\

If $\bfP$ and $\bfQ$ are two species, their Cauchy tensor product is defined as follows:
\begin{itemize}
\item For any finite set $X$,
\[\bfP\otimes \bfQ[X]=\bigoplus_{X=X_1\sqcup X_2}\bfP[X_1]\otimes \bfQ[X_2].\]
\item If $\sigma:X\longrightarrow Y$ is a bijection between two finite sets, then $\bfP\otimes \bfQ[\sigma]$
is defined by
\[\bfP\otimes \bfQ[\sigma]_{\mid \bfP[X_1]\otimes \bfQ[X_2]}
=\bfP[\sigma_{\mid X_1}]\otimes \bfQ[\sigma_{\mid X_2}],\]
for any sets $X_1,X_2$ such that $X=X_1\sqcup X_2$. This takes its values in $\bfP[\sigma(X_1)]\otimes
\bfQ[\sigma(X_2)]\subseteq \bfP\otimes \bfQ[Y]$. 
\end{itemize}
If $f:\bfP\longrightarrow \bfP'$ and $g:\bfQ\longrightarrow \bfQ'$ are species morphism, then 
\[(f\otimes g)_X=\bigoplus_{X=X_1\sqcup X_2}f_{X_1}\otimes g_{X_2}:\bfP\otimes \bfQ[X]\longrightarrow \bfP'\otimes \bfQ'[X].\]
The unit species $\bfI$ is given by
\[\bfI[X]=\begin{cases}
\K\mbox{ if }X=\emptyset,\\
(0)\mbox{ otherwise}.
\end{cases}\]
Moreover, if $\bfP$ and $\bfQ$ are two species, then $\bfP\otimes \bfQ$ and $\bfQ\otimes \bfP$ are
naturally isomorphic, thanks to the flip $c_{\bfP,\bfQ}$, which we will often denote simply by $c$:
\[c_{\bfP,\bfQ}:\left\{\begin{array}{rcl}
\bfP\otimes \bfQ&\longrightarrow &\bfQ\otimes \bfP\\
x\otimes y\in \bfP[X]\otimes \bfQ[Y]\subseteq \bfP\otimes \bfQ[X\sqcup Y]&\longrightarrow&
y\otimes x\in \bfQ[Y]\otimes \bfP[X]\subseteq \bfQ\otimes \bfP[X\sqcup Y].
\end{array}\right.\]
Hence, the category of species, with the Cauchy tensor product, is symmetric monoidal.\\

The category of species has a second tensor product, called the Hadamard tensor product and here denoted by $\boxtimes$.
\begin{itemize}
\item If $\bfP$ and $\bfQ$ are two species, for any finite set $X$,
\[\bfP\boxtimes \bfQ[X]=\bfP[X]\otimes \bfQ[X].\]
\item If $f:X\longrightarrow Y$ is a bijection between two finite sets, then
\[\bfP\boxtimes \bfQ[\sigma]=\bfP[\sigma]\otimes \bfQ[\sigma]:\bfP\boxtimes \bfP[X]\longrightarrow \bfP\boxtimes \bfQ[Y].\]
\end{itemize}
If $f:\bfP\longrightarrow \bfP'$ and $g:\bfQ\longrightarrow \bfQ'$ are species morphisms, then 
\[(f\boxtimes g)_X=f_X\otimes g_X:\bfP\boxtimes \bfQ[X]\longrightarrow \bfP'\boxtimes \bfQ'[X].\]
The unit is the species $\com$, defined as follows. For any finite set $X$, $\com[X]=\K$ and for any bijection $\sigma$
between two finite sets, $\com[\sigma]=\id_\K$. To make things more understandable,
we shall denote the element $1\in \K=\com[X]$ by $1_X$ for any finite set $X$.
Equipped with $\boxtimes$, the category of species is symmetric monoidal. \\

In the sequel, if $\bfP$ is a species, we shall write $\bfP[n]$ instead of $\bfP\big{[}[n]\big{]}$, for any $n\in \N$. 

\begin{remark}
We shall also use species taking their values in the category of sets (they will be called set species)
or taking their values in the category of coalgebras (they will be called coalgebraic species). 
\end{remark}

\begin{defi} \cite{Aguiar2010,Joyal1981,Joyal1986,Patras2004}
A twisted algebra is an associative and unitary algebra in the category of species with the Cauchy tensor product.
In other words, it is a pair $(\bfP,m)$ where $\bfP$ is a species and $m:\bfP\otimes \bfP\longrightarrow \bfP$ 
is a morphism of species such that the following diagram commutes:
\[\xymatrix{\bfP^{\otimes 3} \ar[r]^{m\otimes \id_\bfP}\ar[d]_{\id_\bfP\otimes m}&\bfP^{\otimes 2}\ar[d]^{m}\\
\bfP^{\otimes 2}\ar[r]_{m}&\bfP}\]
Moreover, there exists a morphism of species $\iota_\bfP:\bfI\longrightarrow \bfP$ such that the following diagram commutes:
\[\xymatrix{\bfI\otimes \bfP\ar[r]^{\iota_\bfP \otimes \id_\bfP}\ar[rd]_{\id}&\bfP^{\otimes 2}\ar[d]^m
&\bfP\otimes \bfI \ar[l]_{\id_\bfP\otimes \iota_\bfP}\ar[ld]^{\id}\\
&\bfP&}\]
The unit element of $\bfP$ is $1_\bfP=\iota_\bfP(1)\in \bfP[\emptyset]$. 
We shall say that $(\bfP,m)$ is commutative if $m\circ c_{\bfP,\bfP}=m$.
\end{defi}

In other words, a twisted algebra $\bfP$ comes with maps $m_{X,Y}:\bfP[X]\otimes \bfP[Y]\longrightarrow \bfP[X\sqcup Y]$ 
 for any finite sets $X$ and $Y$, such that:
\begin{itemize}
\item For any bijections $\sigma:X\longrightarrow X'$ and $\tau:Y\longrightarrow Y'$ between finite sets, the
following diagram commutes:
\[\xymatrix{\bfP[X]\otimes \bfP[Y]\ar[rr]^{m_{X,Y}}\ar[d]_{\bfP[\sigma]\otimes \bfP[\tau]}
&&\bfP[X\sqcup Y]\ar[d]^{^\bfP[\sigma\sqcup \tau]}\\
\bfP[X']\otimes \bfP[Y']\ar[rr]_{m_{X',Y'}}&&\bfP[X'\sqcup Y']}\]
where
\[\sigma\sqcup \tau:\left\{\begin{array}{rcl}
X\sqcup Y&\longrightarrow&X'\sqcup Y'\\
x\in X&\longrightarrow&\sigma(x),\\
y\in Y&\longrightarrow&\tau(y).
\end{array}\right.\]
\item For any finite sets $X,Y,Z$, the following diagram commutes:
\[\xymatrix{\bfP[X]\otimes \bfP[Y]\otimes \bfP[Z]\ar[d]_{\id_{\bfP[X]}\otimes m_{Y,Z}}
\ar[rr]^{m_{X,Y}\otimes \id_{\bfP[Z]}}&&\bfP[X\sqcup Y]\otimes \bfP[Z]\ar[d]^{m_{X\sqcup Y,Z}}\\
\bfP[X]\otimes \bfP[Y\sqcup Z]\ar[rr]_{m_{X,Y\sqcup Z}}&&\bfP[X\sqcup Y\sqcup Z]}\]
\item There exists an element $1_\bfP\in \bfP[\emptyset]$ such that for any finite set $X$, for any $x\in \bfP[X]$,
\[m_{\emptyset,X}(1_\bfP\otimes x)=m_{X,\emptyset}(x\otimes 1_\bfP)=x.\]
\end{itemize}
Moreover, $\bfP$ is commutative if, and only if, for any finite sets $X,Y$, in $\bfP[X\sqcup Y]$,
\begin{align*}
&\forall x\in \bfP[X], \:\forall y\in \bfP[Y],&m_{X,Y}(x\otimes y)=m_{Y,X}(y\otimes x).
\end{align*}

Dually:
\begin{defi}
A twisted coalgebra is a coassociative and counitary coalgebra in the category of species.
In other words, it is a pair $(\bfP,\Delta)$ where $\bfP$ is a species and $\Delta:\bfP\longrightarrow \bfP\otimes \bfP$ 
is a morphism of species such that the following diagram commutes:
\[\xymatrix{\bfP\ar[r]^{\Delta}\ar[d]_{\Delta}&\bfP^{\otimes 2}\ar[d]^{\id_\bfP\otimes \Delta}\\
\bfP^{\otimes 2}\ar[r]_{\Delta \otimes \id_\bfP}&\bfP^{\otimes 3}}\]
Moreover, there exists a morphism of species $\varepsilon_\bfP:\bfP\longrightarrow \bfI$ such that the following diagram commutes:
\[\xymatrix{\bfI\otimes \bfP\ar[r]^{\id}&\bfP\ar[d]^\Delta&\bfP\otimes \bfI \ar[l]_{\id}\\
&\bfP^{\otimes 2}\ar[lu]^{\varepsilon_\bfP \otimes \id_\bfP}\ar[ru]_{\id_\bfP\otimes \varepsilon_\bfP}&}\]
We shall say that $(\bfP,\Delta)$ is cocommutative if $c_{\bfP,\bfP}\circ \Delta=\Delta$.
We shall say that the coalgebra $(\bfP,\Delta)$ is connected if $\bfP[\emptyset]$ is one-dimensional.
\end{defi}

In other words, a twisted coalgebra $\bfP$ comes with maps  $\Delta_{X,Y}:\bfP[X\sqcup Y]\longrightarrow \bfP[X]\otimes \bfP[Y]$
for any finite sets $X$ and $Y$, such that:
\begin{itemize}
\item For any bijections $\sigma:X\longrightarrow X'$ and $\tau:Y\longrightarrow Y'$ between finite sets, 
the following diagram commutes:
\[\xymatrix{\bfP[X\sqcup Y]\ar[rr]^{\Delta_{X,Y}}\ar[d]_{\bfP[\sigma\sqcup \tau]}
&&\bfP[X]\otimes \bfP[Y]\ar[d]^{\bfP[\sigma]\otimes \bfP[\tau]}\\
\bfP[X'\sqcup Y']\ar[rr]_{\Delta_{X',Y'}}&&\bfP[X']\otimes \bfP[Y']}\]
\item For any finite sets $X,Y,Z$, the following diagram commutes:
\[\xymatrix{\bfP[X\sqcup Y\sqcup Z]\ar[rr]^{\Delta_{X\sqcup Y,Z}} \ar[d]_{\Delta_{X\sqcup Y,Z}}
&&\bfP[X\sqcup Y]\otimes \bfP[Z]\ar[d]^{\Delta_{X,Y}\otimes \id_{\bfP[Z]}}\\
\bfP[X]\otimes \bfP[Y\sqcup Z]\ar[rr]_{\id_{\bfP[X]}\otimes \Delta_{Y,Z}}&&\bfP[X]\otimes \bfP[Y]\otimes \bfP[Z]}\]
\item There exists a linear map $\varepsilon_\bfP:\bfP[\emptyset]\longrightarrow \K$ 
such that for any finite set $X$, the following diagram commutes:
\[\xymatrix{\bfP[\emptyset]\otimes \bfP[X]\ar[rrd]_{\varepsilon_\bfP\otimes \id_{\bfP[X]}}
&&\bfP[X] \ar[ll]_{\Delta_{\emptyset,X}}  \ar[rr]^{\Delta_{X,\emptyset}}
\ar[d]_{\id_{\bfP[X]}}&&\bfP[X]\otimes \bfP[\emptyset]\ar[lld]^{\hspace{3mm}\id_{\bfP[X]}\otimes \varepsilon_\bfP}\\
&&\bfP[X]&&}\]
\end{itemize}

As in the "classical" case of bialgebras in the category of vector spaces:

\begin{defi}
Let $\bfP$ be a species, both a twisted  algebra $(\bfP,m_\bfP)$ and a twisted coalgebra $(\bfP,\Delta_\bfP)$. 
The following conditions are equivalent:
\begin{enumerate}
\item $\varepsilon_\bfP:\bfP\longrightarrow \bfI$ and $\Delta_\bfP:\bfP\longrightarrow \bfP\otimes \bfP$ 
are algebra morphisms.
\item $\iota_\bfP:\bfI\longrightarrow \bfP$ and $m:\bfP\otimes \bfP\longrightarrow \bfP$ are coalgebra morphisms.
\end{enumerate}
If this holds, we shall say that $(\bfP,m_\bfP,\Delta_\bfP)$ is a  twisted bialgebra, that is to say a bialgebra in the category of species
\cite{Joyal1981,Joyal1986,Patras2004}.
\end{defi}

The compatibility between the product and coproduct gives the commutativity of the following diagram:
if $X$ is a finite set and $X=I\sqcup J=I'\sqcup J'$, 
\[\xymatrix{\bfP[I']\otimes \bfP[J'] \ar[rrr]^{m_{I',J'}}\ar[d]_{\Delta_{I'\cap I,I'\cap J}\otimes
\Delta_{J'\cap I,J'\cap J}}&&&\bfP[I'\sqcup J']=\bfP[I\sqcup J]\ar[dd]^{\Delta_{I,J}}\\
\bfP[I'\cap I]\otimes \bfP[I'\cap J]\otimes \bfP[J'\cap I]\otimes \bfP[J'\cap J]
\ar[d]_{\id_{\bfP[I'\cap I]}\otimes c_{\bfP[I'\cap J],\bfP[J'\cap I]}\otimes \id_{\bfP[J'\cap J]}}&&&\\
\bfP[I'\cap I]\otimes \bfP[J'\cap I]\otimes \bfP[I'\cap J]\otimes \bfP[J'\cap J]
\ar[rrr]_{\hspace{3cm}m_{I'\cap I,J'\cap I}\otimes m_{I'\cap J,J'\cap J}}&&&\bfP[I]\otimes \bfP[J]
}\]
or equivalently
\begin{align*}
\Delta_{I,J}\circ m_{I,'J'}&=(m_{I'\cap I,J'\cap I}\otimes m_{I'\cap J,J'\cap J})\circ (\id_{\bfP[I\cap I]}
\otimes c_{\bfP[I'\cap J],\bfP[J'\cap I]}\otimes \id_{\bfP[J'\cap J]})\\
&\circ (\Delta_{I'\cap I,I'\cap J}\otimes \Delta_{J'\cap I,J'\cap J}).
\end{align*}
The compatibility between the coproduct and the unit is written as $\Delta_\bfP(1_\bfP)=1_\bfP\otimes 1_\bfP$
and the compatibility between the counit and the product is equivalent to
\begin{align*}
&\forall x,y\in \bfP[\emptyset],&\varepsilon_\bfP(xy)=\varepsilon_\bfP(x)\varepsilon_\bfP(y).
\end{align*}

\subsection{The twisted algebra of tensor powers}

Let us fix a vector space $V$ and $X$ a finite set. 
We define the vector space $V^{\otimes X}$ (unordered tensor product indexed by $X$)
 as the vector space generated by sequences $(u_x)_{x\in X}\in V^X$ of elements of $V$ indexed by $X$, with the relations
\[(\lambda_x u_x+\mu_x v_x)_{x\in X}=\sum_{I\subseteq X} \prod_{x\in X}\nu_x^{(I)} \left(w_x^{(I)}\right)_{x\in X},\]
for any $(\lambda_x)_{x\in X}$, $(\mu_x)_{x\in X}\in \K^X$ and $(u_x)_{x\in X}$, $(v_x)_{x\in X}\in V^X$, with the notations
\begin{align*}
\nu_x^{(I)}&=\begin{cases}
\lambda_x\mbox{ if }x\in I,\\
\mu_x\mbox{ if }x\notin I,
\end{cases}&
w_x^{(I)}&=\begin{cases}
u_x\mbox{ if }x\in I,\\
v_x\mbox{ if }x\notin I.
\end{cases}
\end{align*}
The class of $(v_x)_{x\in X}$ in this space is denoted by $\displaystyle \bigotimes_{x\in X} v_x$.
By convention, $\bfT_V[\emptyset]=\K$.  \\

We obtain a species $\bfT_V$:
\begin{itemize}
\item For any finite set $X$, $\bfT_V[X]=V^{\otimes X}$.
\item For any bijection $\sigma:X\longrightarrow Y$ between two finite sets, $\bfT_V[\sigma]:V^{\otimes X}\longrightarrow 
V^{\otimes Y}$ sends $\displaystyle \bigotimes_{x\in X} v_x$ to $\displaystyle \bigotimes_{y\in Y} v_{\sigma^{-1}(y)}$.
\end{itemize}

Note that when $X=[n]$, we obtain the usual tensor product $V^{\otimes n}$, with 
\[ \bigotimes_{i\in [n]} v_i=v_1\otimes \ldots \otimes v_n,\]
which we shall simply denote here by $v_1\ldots v_n$.
The action of $\sym_n$ on $\bfT_V[n]$ induced by the species structure is the usual action on $V^{\otimes n}$ 
by permutations of tensors:
\begin{align*}
&\forall \sigma\in \sym_n,\: \forall v_1,\ldots,v_n \in V,
&\bfT_V[\sigma](v_1\ldots v_n)&=v_{\sigma^{-1}(1)}\ldots v_{\sigma^{-1}(n)}.
\end{align*}
 
\begin{prop}
The species $\bfT_V$ is a commutative and cocommutative twisted bialgebra, with the product and coproduct defined by
\begin{align*}
m_{X,Y}\left(\bigotimes_{x\in X} v_x\otimes \bigotimes_{y\in  Y} v_y\right)
&=\bigotimes_{z\in X\sqcup Y} v_z,&
\Delta_{X,Y}\left(\bigotimes_{z\in X\sqcup Y} v_z\right)&=\bigotimes_{x\in X} v_x\otimes 
\bigotimes_{y\in  Y} v_y.
\end{align*}
The unit is the element $1\in \K=\bfT_V[\emptyset]$ and the counit $\varepsilon_\Delta:\bfT_V[\emptyset]=\K\longrightarrow 
\bfI[\emptyset]=\K$ is the identity of $\K$.
\end{prop}

\begin{proof}
Immediate verifications. 
\end{proof}

\begin{example} \label{ex1.1}
If $V=\K$, up to an isomorphism we obtain the twisted bialgebra $\com$. The twisted products and coproducts are given by
\begin{align*}
m_{X,Y}(1_X\otimes 1_Y)&=1_{X\sqcup Y},&\Delta_{X,Y}(1_{X\sqcup Y})&=1_X\otimes 1_Y.
\end{align*}
\end{example}

\subsection{The twisted bialgebra of graphs}

We refer to \cite{Harary1969} for classical notations and results on graphs. 
A graph is a pair $G=(V(G),E(G))$, where $V(G)$ is a finite set and $E(G)$ is a set made of pairs of elements of $V(G)$.
For any finite set $X$, we denote by $\calG[X]$ the set of graphs $G$ such that $V(G)=X$
and by $\bfG[X]$ the vector space generated by $\calG[X]$. 
If $\sigma:X\longrightarrow Y$ is a bijection between two finite sets, it induces a bijection $\bfG[\sigma]:\bfG[X]
\longrightarrow \bfG[Y]$, sending any graph $G=(X,E(G))$ to the graph $\bfG[\sigma](G)$ defined by
\begin{align*}
V(\bfG[\sigma](G))&=Y,&E(\bfG[\sigma](G))&= \{\sigma(e)\mid e\in E(G)\}.
\end{align*}
This defines a species $\bfG$ and a set species $\calG$. For example,
\begin{align*}
\calG[\{a,b\}]&=\{\tddeux{$a$}{$b$},\tdun{$a$}\tdun{$b$}\},\\
\calG[\{a,b,c\}]&=\{\tdtroisdeux{$a$}{$b$}{$c$},\tdtroisdeux{$b$}{$a$}{$c$},\tdtroisdeux{$a$}{$c$}{$b$},
\gdtroisun{$a$}{$c$}{$b$},\tddeux{$a$}{$b$}\tdun{$c$},\tddeux{$a$}{$c$}\tdun{$b$},
\tddeux{$b$}{$c$}\tdun{$a$},\tdun{$a$}\tdun{$b$}\tdun{$c$}\}.
\end{align*}

If $G$ and $H$ are two graphs, their disjoint union is the graph
\[GH=(V(G)\sqcup V(H), E(G)\sqcup E(H)).\]
This is extended as a product on $\bfG$, making it a twisted algebra. 
The unit is the empty graph $1=(\emptyset,\emptyset)
\in \bfG[\emptyset]$.\\

If $G=(X,E(G))$ is a graph and $A\subseteq X$, we define the induced graph $G_{\mid A}$ by
\begin{align*}
V(G_{\mid A})&=A,&
E(G_{\mid A})&=\{\{x,y\}\in E(G)\mid x,y\in A\}.
\end{align*}
We then define a coproduct on $\bfG$ by putting for any graph $G\in \calG[A\sqcup Y]$,
\[\Delta_{A,Y}(G)=G_{\mid A}\otimes G_{\mid Y}.\]
It is coassociative, and the counit $\varepsilon_\Delta$ send the empty graph $1$ to $1$.

\begin{example} in $\bfG[\{a,b,c\}]$,
\begin{align*}
\Delta_{\{a\},\{b,c\}}(\tdtroisdeux{$a$}{$b$}{$c$})&=\tdun{$a$}\otimes \tddeux{$b$}{$c$},&
\Delta_{\{b,c\},\{a\}}(\tdtroisdeux{$a$}{$b$}{$c$})&=\tddeux{$b$}{$c$}\otimes \tdun{$a$},\\
\Delta_{\{b\},\{a,c\}}(\tdtroisdeux{$a$}{$b$}{$c$})&=\tdun{$b$}\otimes \tdun{$a$}\tdun{$c$},&
\Delta_{\{a,c\},\{b\}}(\tdtroisdeux{$a$}{$b$}{$c$})&=\tdun{$a$}\tdun{$c$}\otimes \tdun{$b$},\\
\Delta_{\{c\},\{a,b\}}(\tdtroisdeux{$a$}{$b$}{$c$})&=\tdun{$c$}\otimes \tddeux{$a$}{$b$},&
\Delta_{\{a,b\},\{c\}}(\tdtroisdeux{$a$}{$b$}{$c$})&=\tddeux{$a$}{$b$}\otimes \tdun{$c$},\\
\Delta_{\{a\},\{b,c\}}(\gdtroisun{$a$}{$c$}{$b$})&=\tdun{$a$}\otimes \tddeux{$b$}{$c$},&
\Delta_{\{b,c\},\{a\}}(\gdtroisun{$a$}{$c$}{$b$})&=\tddeux{$b$}{$c$}\otimes \tdun{$a$},\\
\Delta_{\{b\},\{a,c\}}(\gdtroisun{$a$}{$c$}{$b$})&=\tdun{$b$}\otimes \tddeux{$a$}{$c$},&
\Delta_{\{a,c\},\{b\}}(\gdtroisun{$a$}{$c$}{$b$})&=\tddeux{$a$}{$c$}\otimes \tdun{$b$},\\
\Delta_{\{c\},\{a,b\}}(\gdtroisun{$a$}{$c$}{$b$})&=\tdun{$c$}\otimes \tddeux{$a$}{$b$},&
\Delta_{\{a,b\},\{c\}}(\gdtroisun{$a$}{$c$}{$b$})&=\tddeux{$a$}{$b$}\otimes \tdun{$c$}.
\end{align*} 
\end{example}

\begin{prop}
\cite{Foissy39}
$(\bfG,m,\Delta)$ is a commutative and cocommutative twisted bialgebra.
\end{prop}

\subsection{Double twisted bialgebras}

\begin{defi}
A twisted bialgebra of the second kind is a twisted algebra in the symmetric monoidal (with the Cauchy tensor product) 
category of coalgebraic species, that is to say a triple $(\bfP,m,\delta)$ such that:
\begin{itemize}
\item $(\bfP,m)$ is a twisted algebra, of unit denoted by $1_\bfP$.
\item For any finite set $X$, $\delta_X:\bfP[X]\longrightarrow \bfP[X]\otimes \bfP[X]$ is a coassociative and counitary coproduct
making $\bfP[X]$ a coalgebra, which counit is denoted by $\epsilon_X$. Moreover, for any bijection $\sigma:X\longrightarrow Y$
between two finite sets, the following diagrams commute:
\begin{align*}
&\xymatrix{\bfP[X]\ar[r]^(.4){\delta_X}\ar[d]_{\bfP[\sigma]}&\bfP[X]\otimes \bfP|X] \ar[d]^{\bfP[\sigma]\otimes \bfP[\sigma]}\\
\bfP[Y]\ar[r]_(.4){\delta_Y}&\bfP[Y]\otimes \bfP|Y]}&
&\xymatrix{\bfP[X]\ar[r]^{\epsilon_X}\ar[d]_{\bfP[\sigma]}&\K\ar[d]^{\id_\K}\\
\bfP[Y]\ar[r]_{\epsilon_Y}&\K}
\end{align*}
\item For any finite sets $X$ and $Y$,  the following diagram commutes:
\[\xymatrix{\bfP[X]\otimes \bfP[Y]\ar[rrr]^{m_{X,Y}}\ar[d]_{\delta_X\otimes \delta_Y}
&&&\bfP[X\sqcup Y]\ar[dd]^{\delta_{X\sqcup Y}}\\
\bfP[X]\otimes \bfP[X]\otimes \bfP[Y]\otimes \bfP[Y]\ar[d]_{\id_{\bfP[X]} \otimes c_{\bfP,\bfP}\otimes \id_{\bfP[Y]}}&&&\\
\bfP[X]\otimes \bfP[Y]\otimes \bfP[X]\otimes \bfP[Y]\ar[rrr]_{m_{X,Y}\otimes m_{X,Y}}
&&&\bfP[X\sqcup Y]\otimes \bfP[X\sqcup Y]}\]
\item $\delta_\emptyset(1_\bfP)=1_\bfP\otimes 1_\bfP$. 
\end{itemize}
\end{defi}

These objects can also be interpreted as bialgebras in the symmetric monoidal category of species, with the Hadamard tensor product.

\begin{notation}
Let $(\bfP,m,\delta)$ be a twisted bialgebra of the second kind. For any finite sets $X_1,\ldots,X_k$,
putting $X=X_1\sqcup \ldots \sqcup X_k$, then $\bfP[X_1]\otimes\ldots \otimes \bfP[X_k]$ becomes a right 
$(\bfP[X],\delta_X)$-comodule with the coaction $\rho_{X_1,\ldots,X_k}$ defined by
\[m_{1,3,\ldots,2k-1,24\ldots 2k}\circ (\delta_{X_1}\otimes \ldots \otimes \delta_{X_k}):
\bfP[X_1]\otimes \ldots  \otimes\bfP[X_k]\longrightarrow (\bfP[X_1]\otimes \ldots \otimes \bfP[X_k])\otimes \bfP[X]\]
where
\[m_{1,3,\ldots,24\ldots 2k}:\left\{\begin{array}{rcl}
\bfP[X_1]^{\otimes 2}\otimes \ldots \otimes \bfP[X_k]^{\otimes 2}&\longrightarrow
&\bfP[X_1]\otimes \ldots \otimes \bfP[X_k]\otimes \bfP[X]\\
a_1\otimes b_1\otimes \ldots \otimes a_k\otimes b_k&\longrightarrow&
a_1\otimes \ldots \otimes a_k\otimes m_{X_1,\ldots,X_k}(b_1\otimes \ldots \otimes b_k). 
\end{array}\right.\]
In particular, for any finite set $X$, $\rho_X=\delta_X$. \\
\end{notation}

\begin{remark}
For any finite sets $X$ and $Y$, the product $m_{X,Y}:\bfP[X]\otimes \bfP[Y]\longrightarrow \bfP[X\sqcup Y]$
and the unit $\iota_\bfP:\K\longrightarrow \bfP[\emptyset]$ are comodule morphisms if and only if
\begin{align*}
\delta_{X\sqcup Y}\circ m_{X,Y}&=(m_{X,Y}\otimes \id_{\bfP[X\sqcup Y]})\circ \rho_{X,Y},&
\delta_\emptyset(1_\bfP)&=1_\bfP\otimes 1_\bfP,
\end{align*}
or, equivalently,
\begin{align*}
\delta_{X\sqcup Y}\circ m_{X,Y}&=(m_{X,Y}\otimes m_{X,Y})\circ \delta_{X\sqcup Y},&
\delta_\emptyset(1_\bfP)&=1_\bfP\otimes 1_\bfP,
\end{align*}
that is to say if, and only if, $\delta:\bfP\longrightarrow \bfP\boxtimes \bfP$ is a twisted algebra morphism.
\end{remark}

\begin{defi}
A double twisted bialgebra is a family $(\bfP,m,\Delta,\delta)$ such that:
\begin{itemize}
\item $(\bfP,m,\Delta)$ is a twisted bialgebra, of unit and counit respectively denoted by $1_\bfP$ and $\varepsilon_\Delta$.
\item $(\bfP,m,\delta)$ is a twisted bialgebra of the second kind, of counit $\epsilon_\delta$.
\item The coproduct $\Delta$ and the counit $\varepsilon_\Delta$ are comodule morphisms, that is to say:
\begin{itemize}
\item For any finite sets $X$ and $Y$, the coproduct $\Delta_{X,Y}:\bfP[X\sqcup Y]\longrightarrow \bfP[X]\otimes \bfP[Y]$
is a $\bfP[X\sqcup Y]$-comodule morphism, that is to say
\[\rho_{X,Y}\circ \Delta_{X,Y}=(\Delta_{X,Y}\otimes \id)\circ \delta_{X\sqcup Y}.\]
\item The counit $\varepsilon_\Delta:\bfP[\emptyset]\longrightarrow \K$ is a $\bfP[\emptyset]$-comodule morphism, 
that is to say, for any $x\in \bfP[\emptyset]$,
\[\varepsilon_\Delta(x)1_\bfP=(\varepsilon_\Delta\otimes \id)\circ \delta_\emptyset.\]
\end{itemize}
\end{itemize}
\end{defi}

\begin{example}\label{ex1.2}
The species $\com$  is given a double twisted bialgebra structure:
for any finite disjoint sets $X$ and $Y$,
\begin{align*}
m_{X,Y}(1_X\otimes 1_Y)&=1_{X\sqcup Y},&
\Delta_{X,Y}(1_{X\sqcup Y})&=1_X\otimes 1_Y,&
\delta_X(1_X)&=1_X\otimes 1_X.
\end{align*}
\end{example}

\begin{example}\cite{Foissy39}
For any finite set $X$, let us denote by $\calcomp[X]$ the set of set compositions of $X$, that is to say finite ordered  sequences
$(X_1,\ldots,X_k)$ of nonempty subsets of $X$ such that $X_1\sqcup\ldots \sqcup X_k=X$. This defines a set species.
The vector space generated by $\calcomp[X]$ is denoted by $\comp[X]$, and this defines a species $\comp$.
It is a double twisted bialgebra:
\begin{itemize}
\item For any $(C_1,C_2)\in \calcomp[X]\times \calcomp[Y]$,
\[C_1\squplus C_2=\sum_{\substack{C\in \calcomp[X\sqcup Y],\\ C_{\mid X}=C_1,\: C_{\mid Y}=C_2}}
C,\]
where, if $C=(X_1,\ldots,X_n)\in \calcomp[Y']$ and $X'\subseteq Y'$, 
$C_{\mid X'}$ is obtained from the sequence $(X_1\cap X',\ldots, X_k\cap X')$ by deletion of the empty sets. This is an element
of $\calcomp[X']$. For example, if $X,Y,Z,T$ are finite sets:
\begin{align*}
(X)\squplus (Y)&=(X,Y)+(Y,X)+(X\sqcup Y),\\
(X,Y)\squplus (Z)&=(X,Y,Z)+(X,Z,Y)+(Z,X,Y)+(X\sqcup Y,Z)+(X,Y\sqcup Z),\\
(X,Y)\squplus (Z,T)&=(X,Y,Z,T)+(X,Z,Y,T)+(Z,X,Y,T)\\
&+(X,Z,T,Y)+(Z,X,T,Y)+(Z,T,X,Y)\\
&+(X,Y\sqcup Z,T)+(X\sqcup Z,Y,T)+(X\sqcup Z,T,Y)\\
&+(X,Z,Y\sqcup T)+(Z,X,Y\sqcup T)+(Z,X\sqcup T,Y)+(X\sqcup Z,Y\sqcup T),
\end{align*}
\item If $(X_1,\ldots,X_k)\in \calcomp[X\sqcup Y]$,
\[\Delta_{X,Y}((X_1,\ldots,X_k))=\begin{cases}
(X_1,\ldots,X_i)\otimes (X_{i+1},\ldots, X_k)\\
\hspace{1cm}\mbox{ if there exists a (necessarily unique) $i$}\\
\hspace{1cm}\mbox{ such that $X=X_1\sqcup \ldots X_i$},\\
0\mbox{ otherwise.}
\end{cases}\]
\item If $(X_1,\ldots,X_k)\in \calcomp[X]$,
\begin{align*}
&\delta_X((X_1,\ldots,X_k))\\
&=\sum_{1\leq i_1<\ldots<i_p<k}\left(\bigsqcup_{i=1}^{i_1} X_i,\ldots, \bigsqcup_{i=i_p+1}^k X_i\right)
\otimes (X_1,\ldots,  X_{i_1})\squplus \ldots\squplus (X_{i_p+1},\ldots,  X_k).
\end{align*}
For example, if $X$, $Y$ and $Z$ are finite sets:
\begin{align*}
\delta(X)&=(X)\otimes (X),\\
\delta(X,Y)&=(X,Y)\otimes (X)\squplus (Y)+(X\sqcup Y)\otimes (X,Y),\\
\delta(X,Y,Z)&=(X,Y,Z)\otimes (X)\squplus (Y)\squplus (Z)+(X,Y\sqcup Z)\otimes (X)\squplus (Y,Z)\\
&+(X\sqcup Y,Z)\otimes (X,Y)\squplus (Z)+(X\sqcup Y\sqcup Z)\otimes (X,Y,Z).
\end{align*}
The counit is given by $\epsilon_\delta(X_1,\ldots, X_k)=\delta_{k,1}$ if $k\geq 1$. 
\end{itemize}\end{example}

\subsection{Double twisted algebras of tensor products}

\begin{notation} 
Let $(V,\delta_V)$ be a coalgebra.
We shall use Sweedler's notation $\delta_V(v)=v'\otimes v''$ for this coalgebra.
\end{notation}

\begin{prop}\label{proptensor}
Let $(V,\delta_V)$ be a coalgebra.  Its counit is denoted by $\epsilon_V$. The twisted bialgebra $(\bfT_V,m,\Delta)$ 
is given a structure of double twisted bialgebra with the coproduct defined by
\[\delta_X\left(\bigotimes_{x\in X} v_x\right)
=\left(\bigotimes_{x\in X} v'_x\right)\otimes\left(\bigotimes_{x\in X} v'_x\right).\]
The counit $\epsilon_\delta$ sends $\displaystyle \bigotimes_{x\in X} v_x$ to $\displaystyle \prod_{x\in X}\epsilon_V(v_x)$.
\end{prop}

\begin{proof}
For any finite set $X$, it is immediate that $\bfT_V[X]$ is a coalgebra, isomorphic to  the usual tensor product of $|X|$ copies of $V$.
Let $X$ and $Y$ be two finite sets. Let us prove that $m_{X,Y}$ is a comodule morphism.
\begin{align*}
&\delta_{X\sqcup Y}\circ m_{X,Y}\left(\left(\bigotimes_{x\in X}v_x\right)\otimes \left(\bigotimes_{y\in Y}v_y\right)\right)\\
&=\delta_{X\sqcup Y}\left(\bigotimes_{z\in X\sqcup Y}v_z\right)\\
&=\left(\bigotimes_{z\in X\sqcup Y}v'_z\right)\otimes\left(\bigotimes_{z\in X\sqcup Y}v''_z\right)\\
&=m_{X,Y}\left(\left(\bigotimes_{x\in X}v'_x\right)\otimes \left(\bigotimes_{y\in Y}v'_y\right)\right)
\otimes m_{X',Y'}\left(\left(\bigotimes_{x\in X}v''_x\right)\otimes \left(\bigotimes_{y\in Y}v''_y\right)\right)\\
&=m_{X,Y}\circ \rho_{X,Y}\left(\left(\bigotimes_{x\in X}v_x\right)\otimes \left(\bigotimes_{y\in Y}v_y\right)\right).
\end{align*}
Let us prove that $\Delta_{X,Y}$ is a comodule morphism.
\begin{align*}
\rho_{X,Y}\circ \Delta_{X,Y}\left(\left(\bigotimes_{z\in X\sqcup Y}v_z\right)\right)
&=\rho_{X,Y}\left(\left(\bigotimes_{x\in X}v_x\right)\otimes \left(\bigotimes_{y\in Y}v_y\right)\right)\\
&=\left(\bigotimes_{x\in X}v'_x\right)\otimes \left(\bigotimes_{y\in Y}v'_y\right)\otimes \left(\bigotimes_{z\in X\sqcup Y}v''_z\right)\\
&=(\Delta_{X,Y}\otimes \id) \left(\left(\bigotimes_{z\in X\sqcup Y}v'_z\right)\otimes \left(\bigotimes_{z\in X\sqcup Y}v''_z\right)\right)\\
&=(\Delta_{X,Y}\otimes \id)\circ \delta_{X\sqcup Y}\left(\bigotimes_{z\in X\sqcup Y}v_z\right).
\end{align*}
finally, as $\bfT_V[\emptyset]=\K$ and $\varepsilon_\Delta(1)1_{\bfT_V}=1
=(\varepsilon_\Delta\otimes\id)\circ \delta_\emptyset(1)$,
we obtain that $\varepsilon_\Delta$ is a comodule morphism. \end{proof}

\begin{example}
For $V=\K$,  we obtain again the double twisted bialgebra $\com$ of Example \ref{ex1.2},
with $\displaystyle 1_X=\bigotimes_{x\in X}1$.
\end{example}

\section{Double twisted bialgebras from species with contraction}

\subsection{Composition of species}

\begin{notation}\label{notationequivalence}
\begin{enumerate}
\item Let $X$ be a finite set. We denote by $\eq[X]$ the set of equivalences on $X$. 
if $\sim\in \eq[X]$ and $x\in X$, the equivalence class of $x$ for $\sim$ is denoted by $\cl_\sim(x)$. 
For any bijection $\sigma:X\longrightarrow Y$ between two finite sets,
we define a bijection $\eq[\sigma]:\eq[X]\longrightarrow\eq[Y]$ as follows: if  $\sim\in \eq[X]$,
then $\eq[\sigma](\sim)=\sim_\sigma$ is the equivalence on $Y$ given by
\begin{align*}
&\forall y,y'\in Y,&y\sim_\sigma y'&\Longleftrightarrow \sigma^{-1}(y)\sim \sigma^{-1}(y').
\end{align*}
This defines a set species $\eq$. Moreover, $\sigma$ induces a bijection between $X/\sim$ and $Y/\sim_\sigma$:
\[\sigma/\sim:\left\{\begin{array}{rcl}
X/\sim&\longrightarrow&Y/\sim_\sigma\\
\cl_\sim(x)&\longrightarrow&\cl_{\sim_\sigma}(\sigma(x)).
\end{array}\right.\]
\item The set $\eq[X]$ is partially ordered by the refinement order: 
\begin{align*}
&\forall \sim,\sim'\in \eq[X],&\sim\leq \sim'&\Longleftrightarrow (\forall x,y\in X,\: x\sim' y\Longrightarrow x\sim y),
\end{align*}
or equivalently, $\sim\leq \sim'$ if the classes of $\sim$ are unions of classes of $\sim'$.
\item If $\sim'\in \eq[X]$, there is a bijection
\[\left\{\begin{array}{rcl}
\eq[X/\sim']&\longrightarrow&\{\sim\in \eq[X],\:\sim\leq \sim'\}\\
\overline{\sim}&\longrightarrow&\sim,
\end{array}\right.\]
where for any $x,y\in X$, $x\sim y$ if and only if $\cl_{\sim'}(x)\overline{\sim}\cl_{\sim'}(y)$.
Moreover, the following map is a bijection:
\[\left\{\begin{array}{rcl}
(X/\sim')/\overline{\sim}&\longrightarrow&X/\sim\\
\cl_{\overline{\sim}}(\cl_{\sim'}(x))&\longrightarrow&\cl_\sim(x).
\end{array}\right.\]
From now, we identify $\eq[X/\sim']$ and  $\{\sim\in \eq[X],\:\sim\leq \sim'\}$, as well as
$(X/\sim')/\overline{\sim}$ and $X/\sim$ through these bijections.
\end{enumerate} \end{notation}

If $\bfP$ and $\bfQ$ are two species, their composition \cite{Joyal1981,Joyal1986} is the species $\bfP\circ \bfQ$ defined as follows:
\begin{itemize}
\item If $X$ is a finite set,
\[\bfP\circ \bfQ[X]=\bigoplus_{\sim\in \eq[X]} \bfP[X/\sim]\otimes \left(\bigotimes_{Y\in X/\sim} \bfQ[Y]\right).\]
\item If $\sigma:X\longrightarrow X'$ is a bijection between two finite sets,
\[\bfP\circ \bfQ[\sigma]=\bigoplus_{\sigma\in \eq[X]} \bfP[\sigma/\sim]\otimes \left(\bigotimes_{Y\in X/\sim}
\bfQ[\sigma_{\mid Y}]\right),\]
noticing that 
\[\bfP[\sigma/\sim]\otimes \left(\bigotimes_{Y\in X/\sim}\bfQ[\sigma_{\mid Y}]\right):
\bfP[X/\sim]\otimes \bigotimes_{Y\in X/\sim} \bfQ[Y]\longrightarrow 
\bfP[X'/\sim_\sigma]\otimes \left(\bigotimes_{Y'\in X'/\sim_\sigma} \bfQ[Y']\right).\]
\end{itemize}
The unit species for $\circ$ is the species $\mathbf{I}$ defined by
\[\mathbf{I}[X]=\begin{cases}
\K\mbox{ if }|X|=1,\\
0 \mbox{ otherwise}.
\end{cases}\]
Let $\bfP$, $\bfQ$ and $\bfR$ be three species. For any finite set $X$,
\begin{align*}
(\bfP\otimes \bfQ)\circ \bfR[X]&=\bigoplus_{\substack{\sim \in \eq[X],\\X/\sim=\overline{X'}\sqcup \overline{X''}}}
\bfP[\overline{X'}]\otimes \bfQ[\overline{X''}]\otimes \left(\bigotimes_{Y\in\overline{ X'}\sqcup \overline{X''}} \bfR[Y]\right)\\
&=\bigoplus_{\substack{X=X'\sqcup X'',\\ \sim'\in \eq[X'],\\ \sim''\in \eq[X'']}}
\bfP[X'/\sim']\otimes \bfQ[X''/\sim'']\otimes \left(\bigotimes_{Y'\in X'/\sim'}\bfR[Y']\right)\otimes
\left(\bigotimes_{Y''\in X''/\sim'}\bfR[Y'']\right)\\
&\approx \bigoplus_{\substack{X=X'\sqcup X'',\\ \sim'\in \eq[X'],\\ \sim''\in \eq[X'']}}
\bfP[X'/\sim']\otimes \left(\bigotimes_{Y''\in X''/\sim'}\bfR[Y'']\right)
\otimes \bfQ[X''/\sim'']\otimes \left(\bigotimes_{Y'\in X'/\sim'}\bfR[Y']\right)\\
&\approx \bigoplus_{X=X'\sqcup X''} \left(\bfP\circ\bfR[X']\right)\otimes \left(\bfQ\circ \bfR[X'']\right)\\
&\approx (\bfP\circ \bfR)\otimes (\bfQ\circ \bfR)[X].
\end{align*}
We now identify the vector spaces $(\bfP\otimes \bfQ)\circ \bfR[X]$ and $ (\bfP\circ \bfR)\otimes (\bfQ\circ \bfR)[X]$.
We can prove similarly that
for any bijection $\sigma:X\longrightarrow Y$ between two finite sets,
\[(\bfP\otimes \bfQ)\circ \bfR[\sigma]= (\bfP\circ \bfR)\otimes (\bfQ\circ \bfR)[\sigma].\] 
Hence, the species $(\bfP\otimes \bfQ)\circ \bfR$ and $ (\bfP\circ \bfR)\otimes (\bfQ\circ \bfR)$ are naturally isomorphic,
and we now identify them. For any species $\bfQ$, we obtain an endofunctor $\calF_{\circ\bfQ}$ of the symmetric monoidal
(for the Cauchy tensor product\footnote{In fact, also for the Hadamard product.}) of species, such that for any finite species $\bfP$, 
$\calF_{\circ \bfQ}[\bfP]=\bfP\circ \bfQ$.

\subsection{An endofunctor on twisted bialgebras}

Let us consider the endofunctor $\calF_{\circ \com}$. By definition of the composition of species,
if $\bfP$ is a species,
\[\calF_{\circ \com}[\bfP][X]=\bigoplus_{\sim \in \eq[X]} \bfP[X/\sim] \otimes \left(\bigotimes_{Y\in X/\sim} \K\right)
=\bigoplus_{\sim \in \eq[X]} \bfP[X/\sim].\]
If $\sigma:X\longrightarrow Y$ is a bijection between two finite sets, then
\[\calF_{\circ \com}[\bfP][\sigma]:\bigoplus_{\sim\in \eq[X]}\bfP[X/\sim]\longrightarrow
\bigoplus_{\sim\in \eq[Y]}\bfQ[Y/\sim]\]
is defined by $\calF_{\circ \com}[\bfP][\sigma]_{\mid \bfP[X/\sim]}=\bfP[\sigma/\sim]$, which takes its values
in $\bfQ[Y/\sim_\sigma]$. 
If $f:\bfP\longrightarrow \bfQ$ is a morphism of species, then $\calF_{\circ \com}[f]:\bfP\circ \com\longrightarrow
\bfQ\circ \com$ is defined as follows: for any finite set $X$, for any $\sim\in \eq[X]$,
\[\calF_{\circ \com}[f]_{\mid \bfP[X/\sim]}=f[X/\sim]:\bfP[X/\sim]\subseteq \bfP\circ \com[X]
\longrightarrow \bfQ[X/\sim]\subseteq \bfQ\circ \com[X].\]
As the endofunctor $\calF_{\circ \com}$ is compatible with the Cauchy tensor product $\otimes$:

\begin{prop} \label{propendofunctor}
Let $\bfP$ be a species and $\bfP'=\bfP\circ \com$.
\begin{enumerate}
\item If $(\bfP,m)$ is a twisted algebra, then $\bfP'$ is also a twisted bialgebra, with the product defined as follows:
for any finite sets $X$, $Y$, for any $\sim_X\in \eq[X]$ and $\sim_Y\in \eq[Y]$,
\[(m'_{X,Y})_{\mid \bfP[X/\sim_X]\otimes \bfP[Y/\sim_Y]}=m_{X/\sim_X,Y/\sim_Y},\]
which takes its values in 
\[\bfP[X/\sim_X\sqcup Y/\sim_Y]=\bfP[(X\sqcup Y)/(\sim_X\sqcup \sim_Y)]\subseteq \bfP'[X\sqcup Y].\]
The unit is $1_\bfP'=1_\bfP\in \bfP[\emptyset]=\bfP'[\emptyset]$. 
\item If $(\bfP,\Delta)$ is a twisted coalgebra, then $\bfP'$ is also a twisted coalgebra, with the product defined as follows:
for any finite sets $X$ and $Y$, for any $\sim\in \eq[X\sqcup Y]$, putting $\sim_X=\sim\cap X^2$
and $\sim_Y=\sim\cap Y^2$,
\[(\Delta'_{X,Y})_{\mid \bfP[(X\sqcup Y)/\sim]}
=\begin{cases}
\Delta_{X/\sim_X,Y/\sim_Y}\mbox{ if }\sim=\sim_X\sqcup \sim_Y,\\
0\mbox{ otherwise},
\end{cases}\]
which takes its values in $\bfP[X/\sim_X]\otimes \bfP[Y/\sim_Y]\subseteq \bfP'[X]\otimes \bfP'[Y]$.
The counit of $(\bfP',\Delta')$  is the counit of $(\bfP,\Delta)$.
\item If $\bfP$ is a twisted bialgebra, these product and coproduct make $\bfP'$ a twisted bialgebra.
\end{enumerate}
\end{prop}

\begin{proof}
\begin{enumerate}
\item By construction, $m'=\calF_{\circ \com}[m]$ and $\iota_\bfP'=\calF_{\circ \com}[\iota_\bfP]$.
As $\calF_{\circ \com}$ is an endofunctor compatible with the Cauchy tensor product, 
$(\bfP',m')$ is a twisted algebra, of unit $\iota_\bfP'$.
\item By construction, $\Delta'=\calF_{\circ \com}[\Delta]$ and $\varepsilon_\Delta=\calF_{\circ \com}[\varepsilon_\Delta]$.
As $\calF_{\circ \com}$ is an endofunctor compatible with the Cauchy tensor product, 
$(\bfP',\Delta')$ is a twisted algebra, of counit $\varepsilon_\Delta$.
\item Same arguments. \qedhere
\end{enumerate}\end{proof}

\begin{remark}
If $\bfP$ is a twisted algebra (resp. coalgebra, bialgebra)
then $\bfP$ is a twisted subalgebra (resp. subcoalgebra, subbialgebra) of $\bfP\circ \com$.
\end{remark}

\subsection{Species with  contractions and extractions}

\begin{defi} \label{deficontraction}
Let $\bfP$ be a species. A contraction-extraction coproduct on $\bfP$ is a family $\delta$ of maps such that:
\begin{itemize}
\item For any finite set $X$, for any $\sim\in \eq[X]$, $\delta_\sim:\bfP[X]\longrightarrow \bfP[X/\sim]\otimes \bfP[X]$.
\item For any bijection $\sigma:X\longrightarrow Y$ between two finite sets, for any $\sim\in \eq[X]$,
the following diagram commutes:
\[\xymatrix{\bfP[X]\ar[r]^<(.2){\delta_\sim}\ar[d]_{\bfP[\sigma]}&\bfP[X/\sim]\otimes \bfP[X]
\ar[d]^{\bfP[\sigma/\sim]\otimes \bfP[\sigma]}\\
\bfP[Y]\ar[r]_<(.2){\delta_{\sim_\sigma}}&\bfP[Y/\sim_\sigma]\otimes \bfP[Y]}\]
\item If $X$ is a finite set and $\sim\leqslant \sim'\in \eq[X]$, the following diagram commutes,
with Notation \ref{notationequivalence}:
\[\xymatrix{\bfP[X]\ar[r]^<(.3){\delta_{\sim'}} \ar[d]_{\delta_\sim}&\bfP[X/\sim']\otimes \bfP[X]\ar[d]^{\delta_\sim\otimes \id}\\
\bfP[X/\sim]\otimes \bfP[X]\ar[r]_<(0.15){\id \otimes \delta_{\sim'}}&\bfP[X/\sim]\otimes \bfP[X/\sim']\otimes \bfP[X]}\]
If $\sim,\sim'\in \eq[X]$, such that we do not have $\sim\leq \sim'$, then $(\id \otimes \delta_{\sim'})\circ \delta_\sim=0$.
\item There exists a species morphism $\epsilon_\delta:\bfP\longrightarrow \com$ such that for any finite set $X$,
for any $\sim\in \eq[X]$,
\[(\id \otimes \epsilon_\delta[X])\circ \delta_\sim=\begin{cases}
\id_{\bfP[X]}\mbox{ if $\sim$ is the equality of $X$},\\
0\mbox{ otherwise},
\end{cases}\]
and
\[\sum_{\sim \in \eq[X]} (\epsilon_\delta[X/\sim]\otimes \id)\circ \delta_\sim=\id_{\bfP[X]}.\]
\end{itemize}
\end{defi}

\begin{prop}
Let $\bfP$ be a species with a contraction-extraction coproduct $\delta$. We put $\bfP'=\bfP\circ \com$.
For any finite set $X$, we define a coproduct
\[\delta'[X]:\bfP'[X]=\bigoplus_{\sim\in \eq[X]}\bfP[X/\sim]\longrightarrow \bfP'[X]\otimes \bfP'[X]\]
by
\[\delta'[X]_{\mid \bfP[X/\sim']}=\bigoplus_{\sim\leq \sim'}\delta_{\sim}:
\bfP[X/\sim']\longrightarrow \bigoplus_{\sim\leqslant \sim'} \bfP[X/\sim]\otimes \bfP[X/\sim']
\subseteq \bfP'[X]\otimes \bfP'[X].\]
Then $\bfP'$ is a coalgebra in the category of coalgebraic species.
\end{prop}

\begin{proof}
Let us first prove that $\delta':\bfP'\longrightarrow\bfP'\otimes \bfP'$ is a species morphism.
Let $\sigma:X\longrightarrow Y$ be a bijection between two finite sets. For any $\sim'\in \eq[X]$,
\begin{align*}
\delta'[Y]\circ \bfP'[\sigma]_{\mid \bfP[X/\sim']}&=\bigoplus_{\sim\leq \sim'}\delta_{\sim_{\sigma}}\circ \bfP[\sigma/\sim']\\
&=\bigoplus_{\sim\leq \sim'}(\bfP[\sigma/\sim]\otimes \bfP[\sigma/\sim'])\circ \delta_\sigma\\
&=(\bfP'[\sigma]\otimes \bfP'[\sigma])\circ \delta'[X]_{\mid \bfP[X/\sim']}.
\end{align*}
Let us now prove that $\delta'[X]$ is coassociative. Let $X$ be a finite set and $\sim''\in \eq[X]$.
\begin{align*}
(\delta'[X]\otimes \id)\circ \delta'[X]_{\mid \bfP[X/\sim'']}&=\bigoplus_{\sim'\leq \sim''} 
(\delta[X] \otimes \id)\circ \delta_{\sim'}\\
&=\bigoplus_{\sim\leq \sim'\leq \sim''}(\delta_\sim\otimes\id)\circ \delta_{\sim'}\\
&=\bigoplus_{\sim\leq \sim'\leq\sim''}(\id \otimes \delta_{\sim'})\circ \delta_{\sim}\\
&=\bigoplus_{\sim,\sim'\leq\sim''}(\id \otimes \delta_{\sim'})\circ \delta_{\sim}\\
&=\bigoplus_{\sim'\leq \sim''}(\id \otimes \delta_{\sim'})\circ \delta[X]_{\mid \bfP[X/\sim'']}\\
&=(\id\otimes \delta'[X])\circ \delta'[X]_{\mid \bfP[X/\sim'']}.
\end{align*}
For the fourth equality, recall that $(\id \otimes \delta_{\sim'})\circ \delta_{\sim}=0$ if we do not have $\sim\leq \sim'$. \\

Let us finally prove that $\epsilon_\delta$ is a counit. Let $X$ be a finite set and $\sim'\in \eq[X]$.
\begin{align*}
(\epsilon_\delta[X]\otimes \id)\circ \delta'[X]_{\mid \bfP[X/\sim']}
&=\sum_{\sim\leq \sim'} (\epsilon_\delta[X]\otimes \id)\circ \delta_\sim\\
&=\sum_{\sim\in \eq[X/\sim']} (\epsilon_\delta[X]\otimes \id)\circ \delta_\sim\\
&=\id_{\bfP[X/\sim']},\\
(\id \otimes \varepsilon_\delta[X])\circ \delta'[X]_{\mid \bfP[X/\sim']}
&=\sum_{\sim\in \eq[X/\sim']} (\id \otimes \epsilon_\delta[X])\circ \delta_\sim\\
&=(\id \otimes \epsilon_\delta[X])\circ \delta_{=_{\bfP[X/\sim']}}+0\\
&=\id_{\bfP[X/\sim']}. 
\end{align*}
If $\sigma:X\longrightarrow Y$ is a bijection between two finite sets, the compatibility of $\delta$ with $\bfP[\sigma]$
implies that $\bfP[\sigma]$ is a coalgebra isomorphism from $(\bfP[X],\delta_X)$ to $(\bfP[Y],\delta_Y)$.
So $\bfP$ is a coalgebraic species. \end{proof}

\begin{prop} \label{propalgebra}
Let $(\bfP,m)$ be a twisted algebra with a contraction-extraction coproduct $\delta$. 
We assume that for any finite set $X$, for any $\sim\in \eq[X]$, putting $\sim_X=\sim\cap X^2$ and $\sim_Y=\sim\cap Y^2$,
\begin{align*}
\delta_\sim\circ m_{X,Y}&=\begin{cases}
(m_{X/\sim_X,Y/\sim_Y}\otimes m_{X,Y})\circ (\id \otimes c \otimes \id)\circ (\delta_{\sim_X}\otimes \delta_{\sim_Y})
\mbox{ if }\sim=\sim_X\sqcup \sim_Y,\\
0\mbox{ otherwise},
\end{cases}
\end{align*}
and that $\delta_{\sim_\emptyset}(1_\bfP)=1_\bfP\otimes 1_\bfP$, where $\sim_\emptyset$ is the unique equivalence on $\emptyset$.
We put $\bfP'=\bfP\circ \com$ and we give it its structure of twisted algebra $(\bfP',m')$ and its coproduct $\delta'$ 
associated to the contraction-extraction coproduct. Then $(\bfP',m')$ is an algebra in the category of  coalgebraic species.
\end{prop}

\begin{proof}
Let us firstly prove that the unit of $\bfP'$ is a coalgebra morphism.
\begin{align*}
\delta'[\emptyset](1_\bfP')&=\delta'[\emptyset](1_\bfP)=1_\bfP\otimes 1_\bfP=1_\bfP'\otimes 1_\bfP'.
\end{align*}
Let us now prove that the product $m'_{X,Y}$ is a coalgebra morphism from $\bfP'[X]\otimes \bfP'[Y]$ to $\bfP'[X\sqcup Y]$
for any finite sets $X$ and $Y$. Let $\sim'_X\in \eq[X]$ and $\sim'_Y\in \eq[Y]$. 
We put $\sim'=\sim'_X\sqcup \sim'_Y\in \eq[X\sqcup Y]$. then
\begin{align*}
(\delta'[X\sqcup Y]\circ m'_{X,Y})_{\mid \bfP[X/\sim'_X]\otimes \bfP[Y/\sim'_Y]}
&=\delta'[X\sqcup Y]\circ m_{X/\sim'_X,Y/\sim'_Y}\\
&=\sum_{\sim\leq \sim'_X\sqcup \sim'_Y} \delta_\sim\circ m_{X/\sim'_X,Y/\sim'_Y}.
\end{align*}
If $\sim\leq \sim'_X\sqcup \sim'_Y$, let us put $\sim_X=\sim\cap X^2$ and $\sim_Y=\sim\cap Y^2$.
If $\sim\neq \sim_X\sqcup \sim_Y$, then $\delta_\sim\circ m_{X/\sim'_X,Y/\sim'_Y}=0$.
Moreover, $\sim_X\leq \sim'_X$ and $\sim_Y\leq \sim'_Y$; conversely, if $\sim_X\leq \sim'_X$ and $\sim_Y\leq \sim'_Y$,
then $\sim_X\sqcup \sim_Y\leq \sim$. We obtain
\begin{align*}
&(\delta'[X\sqcup Y]\circ m'_{X,Y})_{\mid \bfP[X/\sim'_X]\otimes \bfP[Y/\sim'_Y]}\\
&=\sum_{\substack{\sim_X\leq \sim'_X,\\ \sim_Y\leq \sim'_Y}} \delta_{\sim_X\sqcup \sim_Y}\circ m_{X/\sim'_X,Y/\sim'_Y}\\
&=\sum_{\substack{\sim_X\leq \sim'_X,\\ \sim_Y\leq \sim'_Y}} (m_{X/\sim_X,Y/\sim_Y}\otimes m_{X/\sim'_X,Y/\sim'_Y})
\circ (\id \otimes c \otimes \id)\circ (\delta_{\sim_X}\otimes \delta_{\sim_Y})\\
&=(m'_{X,Y}\otimes m'_{X,Y})\circ (\id \otimes c \otimes \id)\circ 
(\delta'[X]\otimes \delta'[Y])_{\mid \bfP[X/\sim'_X]\otimes \bfP[Y/\sim'_Y]}. \qedhere
\end{align*}
\end{proof}

\begin{prop}\label{propbialgebra}
Let $(\bfP,m,\Delta)$ be a twisted bialgebra of unit $1_\bfP$
and of counit $\varepsilon_\Delta$, with a contraction-extraction coproduct $\delta$.  
We assume that for any finite set $X\sqcup Y$, for any $\sim\in \eq[X\sqcup Y]$, putting $\sim_X=\sim\cap X^2$ 
and $\sim_Y=\sim\cap Y^2$,
\begin{align*}
\delta_\sim\circ m_{X,Y}&=\begin{cases}
(m_{X/\sim_X,Y/\sim_Y}\otimes m_{X,Y})\circ (\id \otimes c \otimes \id)\circ (\delta_{\sim_X}\otimes \delta_{\sim_Y})\\
\hspace{1cm}
\mbox{ if }\sim=\sim_X\sqcup \sim_Y,\\
0\mbox{ otherwise},
\end{cases}\\
(\Delta_{X/\sim_X,Y/\sim_Y}\otimes \id)\circ \delta_\sim&=
m_{1,3,24}\circ (\delta_{\sim_X}\otimes \delta_{\sim_Y})\circ \Delta_{X,Y}\mbox{ if }\sim=\sim_X\sqcup \sim_Y,\\
\delta_{\sim_\emptyset}(1_\bfP)&=1_\bfP\otimes 1_\bfP,\\
(\varepsilon_\Delta\otimes \id)\circ \delta_{\sim_\emptyset}&=\eta_\bfP\circ \varepsilon_\Delta(x).
\end{align*}
Let $\bfP'=\bfP\circ \com$, with its structure of twisted bialgebra $(\bfP',m',\Delta')$ and its coproduct $\delta'$ 
associated to the contraction-extraction coproduct. Then $(\bfP',m,\Delta,\delta)$ is a double twisted bialgebra.
\end{prop}

\begin{remark}
Note that if $\sim\neq \sim_X\sqcup \sim_Y$, then $X/\sim_X$ and $Y/\sim_Y$ are not disjoints subsets of $(X\sqcup Y)/\sim$
and then $(\Delta_{X/\sim_X,Y/\sim_Y}\otimes \id)\circ \delta_\sim$ has no meaning.
\end{remark}

\begin{proof}
We already obtained in Proposition \ref{propalgebra} that $(\bfP',m',\Delta')$ is an algebra  in the category of  coalgebraic species.
The condition on $\varepsilon_\Delta$ proves that $\varepsilon_\Delta:\bfP'[\emptyset]=\bfP[\emptyset]\longrightarrow \K$
is a comodule morphism. Let us prove that for any finite sets $X$ and $Y$, $\Delta'_{X,Y}$ is a comodule morphism.
Let $\sim'\in \eq[X\sqcup Y]$. We put $\sim'_X=\sim'\cap X^2$ and $\sim'_Y=\sim'\cap Y^2$. 
\begin{align*}
(\Delta'_{X,Y}\otimes \id)\circ \delta'[X\sqcup Y]_{\mid\bfP[X\sqcup Y/\sim']}
&=\sum_{\sim\leqslant \sim'} (\Delta'_{X,Y}\otimes \id)\circ \delta_\sim\\
&=\sum_{\substack{\sim_X\in\eq[X],\\ \sim_Y\in \eq[Y],\\ \sim_X\sqcup \sim_Y\leq \sim'}}
(\Delta_{X/\sim_X,Y/\sim_Y}\otimes \id)\circ \delta_{\sim_X\sqcup \sim_Y}\\
&=\sum_{\substack{\sim_X\in\eq[X],\\ \sim_Y\in \eq[Y],\\ \sim_X\sqcup \sim_Y\leq \sim'}}
m_{1,3,24}\circ (\delta_{\sim_X}\otimes \delta_{\sim_Y})\circ \Delta_{X/\sim_X,Y/\sim_Y}.
\end{align*}
If $\sim'\neq \sim'_X\sqcup \sim'_Y$, then $(\Delta'_{X,Y})_{\mid \bfP[X \sqcup Y/\sim']}=0$.
At least one class of $\sim'$ contains elements of $X$ and elements of $Y$; if $\sim \leq \sim'$,
the classes of $\sim$ are union of classes of $\sim'$, so at least one of them contains elements of $X$ and elements of $Y$.
So $\sim\neq \sim_X\sqcup \sim_Y$. We obtain
\begin{align*}
(\Delta'_{X,Y}\otimes \id)\circ \delta'[X\sqcup Y]_{\mid\bfP[X\sqcup Y/\sim']}&=0
=(\rho_{X,Y}\circ \Delta'_{X,Y})_{\mid \bfP[X \sqcup Y/\sim']}.
\end{align*}
If $\sim'=\sim'_X\sqcup \sim'_Y$, then
\begin{align*}
(\Delta'_{X,Y}\otimes \id)\circ \delta'[X\sqcup Y]_{\mid\bfP[X\sqcup Y/\sim']}
&=\sum_{\substack{\sim_X\in\eq[X],\\ \sim_Y\in \eq[Y],\\ \sim_X\sqcup \sim_Y\leq \sim'}}
m_{1,3,24}\circ (\delta_{\sim_X}\otimes \delta_{\sim_Y})\circ \Delta_{X/\sim_X,Y/\sim_Y}\\
&=\sum_{\substack{\sim_X\leq \sim'_X,\\ \sim_Y\leq \sim'_Y}}
m_{1,3,24}\circ (\delta_{\sim_X}\otimes \delta_{\sim_Y})\circ \Delta_{X/\sim_X,Y/\sim_Y}\\
&=(m_{1,3,24}\circ (\delta[X]\otimes \delta[Y])\circ \Delta'_{X,Y})_{\mid \bfP[X\sqcup Y/\sim']}\\
&=(\rho_{X,Y}\circ \Delta'_{X,Y})_{\mid \bfP[X\sqcup Y/\sim']}. \qedhere
\end{align*}
\end{proof}

\subsection{Links with cooperads}

In the species setting, a cooperad is a couple $(\bfP,\delta)$ where $\bfP$ is a species and 
$\delta:\bfP\longrightarrow \bfP\circ \bfP$ is a counitary and coassociative coproduct, that is to say:
\begin{itemize}
\item The following diagram commutes:
\[\xymatrix{\bfP\ar[r]^\delta \ar[d]_\delta &\bfP\circ \bfP\ar[d]^{\id \circ \delta}\\
\bfP\circ\bfP\ar[r]_{\delta \circ \id}&\bfP\circ \bfP\circ \bfP}\]
\item There exists a species morphism $\varepsilon:\bfP\longrightarrow \mathbf{I}$ such that the following diagram commutes:
\[\xymatrix{\bfP\circ \bfP \ar[rd]_{\varepsilon\circ \id}&\bfP\ar[l]_{\delta} \ar[r]^\delta\ar[d]^{\id}
&\bfP\circ \bfP\ar[ld]^{\id \circ \varepsilon}\\
&\bfP&}\]
\end{itemize}

Let $(\bfP,\delta)$ is a cooperad, such that $\bfP[\emptyset]=(0)$. We consider the twisted algebra $S(\bfP)$.
We then consider the map $(\id \otimes m)\circ \delta:\bfP\longrightarrow (\bfP\circ \com)\otimes S(\bfP)$
and extend it multiplicatively to $S(\bfP)$. The coassociativity of $\delta$ and the existence of its counit implies that 
the result is a contraction-extraction coproduct on $S(\bfP)$.

\subsection{Contraction-extraction on graphs}

\begin{defi}
Let $G\in \calG[X]$ and $\sim\in \eq[X]$. 
\begin{enumerate}
\item We define a graph $G\mid\sim\in \calG[X]$ by
\begin{align*}
V(G\mid\sim)&=V(G),\\
E(G\mid\sim)&=\{\{x,y\}\in E(G)\mid\: x\sim y\}.
\end{align*}
In other words, $G\mid \sim$ is obtained from $G$ by deleting all the edges  which extremities are not equivalent;
or equivalently, $G\mid\sim$ is the disjoint union of the restrictions of $G$ to the equivalence classes of $\sim$.
\item We define a graph $G/\sim\in \calG[X/\sim]$ by
\begin{align*}
V(G/\sim)&=X/\sim,\\
E(G/\sim)&=\{\{\cl_\sim(x),\cl_\sim(y)\}\mid\: \{x,y\}\in E(G)\: \cl_\sim(x)\neq \cl_\sim(y)\}.
\end{align*}
In other words, $G/\sim$ is obtained from $G$ by identifying the vertices according to $\sim$,
then deleting the loops created in the process and the redundant edges.
\item We shall say that $\sim\in \eq_c[G]$ if for any  class $C$ of $\sim$, $G_{\mid C}$ is connected.
\end{enumerate}
\end{defi}

\begin{prop} We define a contraction-extraction coproduct $\delta$ on $\bfG$ as follows: for any finite set $X$,
for any $\sim\in \eq[X]$, for any $G\in \calG[X]$,
\[\delta_\sim(G)=\begin{cases}
G/\sim\otimes G\mid \sim \mbox{ if }\sim\in \eq_c[G],\\
0\mbox{ otherwise}.
\end{cases}\]
It is compatible with the product and the coproduct in the sense of Proposition \ref{propbialgebra}.
\end{prop}

\begin{proof}
The compatibility of $\delta$ with the species structure is clear.
Let us prove the coassociativity of $\delta$. Let $X$ be a finite set, $\sim,\sim'\in \eq[X]$ and $G\in \calG[X]$.\\

If $\sim\leq \sim'$, let us prove that $\sim \in \eq_c[G/\sim']$ and $\sim'\in \eq_c[G]$ if, and only if,
$\sim'\in \eq_c[G\mid \sim]$ and $\sim \in \eq_c[G]$.

$\Longrightarrow$. Let $C'$ be a class of $\sim'$. As $\sim'\in \eq_c[G]$, it is a connected subgraph of $G$.
Moreover, as $\sim\leq \sim'$, all its elements are in the same class of $\sim$, so $G_{\mid C'}=(G|\sim)_{\mid C'}$:
as a consequence, $(G|\sim)_{\mid C'}$ is connected, so $\sim'\in \eq_c[G\mid \sim]$. 
Let $C$ be a class of $\sim$, and $x,y\in C$. As $\sim \in \eq_c[G/\sim']$, it is connected in $G/\sim'$: 
there exists a path in $G/\sim'$ from $\cl_{\sim'}(x)$ to $\cl_{\sim'}(y)$. Moreover, as $\sim'\in \eq_c[G]$,
each $\cl_{\sim'}(z)$ is a connected subgraph of $G$, so there is a path from $x$ to $y$ in $G$: $\sim \in \eq_c[G]$.

$\Longleftarrow$. Let $C$ be a class of $\sim$. As $\sim\in \eq_c[G]$, any of its class is a connected subgraph of $G$,
so by contraction is a connected subgraph of $G/\sim'$: $\sim \in \eq_c[G/\sim']$. Let $C'$ be a class of $\sim'$.
As $\sim'\in \eq_c[G\mid \sim]$, it is a connected subgraph of $G\mid \sim$, so also of $G$: $\sim \in \eq_c[G/\sim']$.\\

As a conclusion,
\begin{align*}
(\delta_\sim\otimes \id)\circ \delta_{\sim'}(G)&=\begin{cases}
(G/\sim')/\sim\otimes (G/\sim')\mid \sim\otimes G\mid \sim'\mbox{ if $\sim \in \eq_c[G/\sim']$ and $\sim'\in \eq_c[G]$},\\
0\mbox{ otherwise}
\end{cases}\\
&=\begin{cases}
G/\sim\otimes (G\mid \sim)/\sim'\otimes (G\mid \sim)\mid \sim'
\mbox{ if $\sim'\in \eq_c[G\mid \sim]$ and $\sim \in \eq_c[G]$},\\
0\mbox{ otherwise}
\end{cases}\\
&=(\id \otimes \delta_{\sim'})\circ \delta_\sim(G).
\end{align*}

If we do not  have $\sim\leq \sim'$, then at least one class $C$ of $\sim$ intersects two classes of $\sim'$,
so intersects two connected components of $\mid \sim'$: we obtain that $\sim\notin \eq_c[G\mid \sim']$. So
$\delta_{\sim}(G\mid \sim')=0$ and finally $(\id \otimes \delta_\sim)\circ \delta_{\sim'}(G)=0$.\\

Let us now study the counity. We define a species morphism $\epsilon_\delta:\bfG\longrightarrow\com$ by 
\begin{align*}
&\forall G\in \calG[X],&\epsilon_\delta[X](G)=\begin{cases}
1\mbox{ if }E(G)=\emptyset,\\
0\mbox{ otherwise.}
\end{cases}
\end{align*}
Let $G\in \calG[X]$ and $\sim \in \eq[X]$. If $\sim$ is the equality of $X$,
then $\sim\in \eq_c[G]$, $G/\sim=G$ and $G\mid \sim$ as no edge, so $(\id \otimes \epsilon_\delta[X])\circ \delta_\sim(G)=G$.
Otherwise, either $G\notin \eq_c[G]$ or at least one class of $\sim$ contains an edge, so 
$\epsilon_\delta[X](G\mid \sim)=0$. In both cases, $(\id \otimes \epsilon_\delta[X])\circ \delta_\sim(G)=0$.

Let $\sim\in \eq_c[G]$, such that $E(G/\sim)=\emptyset$. If two vertices of $G$ are related by an edge,
there are necessarily equivalent, so any connected component of $G$ is included in a single class of $\sim$.
As the classes of $\sim$ are connected, $\sim$ is the relation $\sim_c$ which classes are the connected components of $G$.
Moreover, $G/\sim_c$ has no edge and $G\mid \sim_c=G$. Therefore,
\begin{align*}
\sum_{\sim\in \eq[X]}(\epsilon_\delta[X/\sim]\otimes \id)\circ \delta_\sim(G)
&=\sum_{\sim\in \eq_c[G]}(\epsilon_\delta[X/\sim]\otimes \id)\circ \delta_\sim(G)\\
&=(\epsilon_\delta[X/\sim]\otimes \id)\circ \delta_{\sim_c}(G)\\
&=G\mid \sim_c\\
&=G.
\end{align*}

Let us prove the compatibility of $\delta$ with the product. Let $X$ and $Y$ be two finite sets, $\sim\in \eq[X\sqcup Y]$,
$G\in \calG[X]$ and $H\in \calG[Y]$. If $\sim\neq \sim_X\sqcup \sim_Y$, at least one class $C$ of $\sim$ intersects
both $X$ and $Y$, so is not connected in $GH=m_{X,Y}(G\otimes H)$. Therefore, $\sim\notin \eq_c[GH]$ and
\[\delta_\sim\circ m_{X,Y}(G\otimes H)=0.\]
Let us assume that  $\sim=\sim_X\sqcup \sim_Y$. Then $\sim\in \eq_c[GH]$ if, and only if, $\sim_X\in \eq_c[G]$
and $\sim_Y \in \eq_c[H]$, as the connected components of $GH$ are the connected components of $G$ and of $H$.
If so, $(GH)/\sim=(G/\sim_X)(H/\sim_Y)$ and $(GH)\mid \sim=(G\mid \sim_X)(H\mid \sim_Y)$. Therefore,
\begin{align*}
\delta_\sim\circ m_{X,Y}(G\otimes H)&=\begin{cases}
(GH)/\sim\otimes (GH)\mid \sim\mbox{ if }\sim\in \eq_c[GH],\\
0\mbox{ otherwise}
\end{cases}\\
&=\begin{cases}
(G/\sim_X)(H/\sim_Y)\otimes (G\mid \sim_X)(H\mid \sim_Y)\mbox{ if }\sim_X\in \eq_c[G]\mbox{ and }\sim_Y \in \eq_c[H],\\
0\mbox{ otherwise}
\end{cases}\\
&=(m_{X/\sim_X,Y/\sim_Y}\otimes m_{X,Y})\circ (\id \otimes c\otimes \id)\circ (\delta_{\sim_X}\otimes \delta_{\sim_Y})
(G\otimes H).
\end{align*}

Let us finally prove the compatibility of $\delta$ with the coproduct $\Delta$. Let $X$ and $Y$ be two finite sets,
$\sim_X\in \eq[X]$, $\sim_Y\in \eq[Y]$ and $G\in \calG[X]$. We put $\sim=\sim_X\sqcup \sim_Y$.  
Then
\begin{align*}
(\Delta_{X/\sim_X,Y/\sim_Y}\otimes \id)\circ \delta_{\sim}(G) 
&=\begin{cases}
(G/\sim)_{\mid X/\sim_X}\otimes (G/\sim)_{\mid Y/\sim_Y}\otimes G\mid \sim
 \mbox{ if $\sim\in \eq_c[G]$},\\
0\mbox{ otherwise};
\end{cases}\\
m_{1,3,24}\circ (\delta_{\sim_X}\otimes \delta_{\sim_Y})\circ \Delta_{X,Y}(G)
&=\begin{cases}
(G_{\mid X})/\sim_X\otimes (G_{\mid Y})/\sim_Y
\otimes (G_{\mid X})\mid \sim_X (G_{\mid Y})\mid \sim_Y\\
\hspace{.5cm}
\mbox{ if $\sim_X\in \eq_c[G_{\mid X}]$ and $\sim_Y\in \eq_c[G_{\mid Y}]$},\\
0\mbox{ otherwise}.
\end{cases}
\end{align*}
As $\sim=\sim_X\sqcup \sim_Y$, $\sim\in \eq_c[G]$ if and only if $\sim_X\in \eq_c[G_{\mid X}]$ and $\sim_Y\in \eq_c[G_{\mid Y}]$. Moreover, 
\begin{align*}
(G/\sim)_{\mid X/\sim_X}&=(G_{\mid X})/\sim_X,&(G/\sim)_{\mid Y/\sim_Y}&=(G_{\mid Y})/\sim_Y,&
G\mid \sim&=(G_{\mid X})\mid \sim_X (G_{\mid Y})\mid \sim_Y,
\end{align*}
which finally proves the compatibility between $\delta$ and $\Delta$. \end{proof}

\begin{example} Let us apply our machinery to the species of graphs $\bfG$. 
For any finite set $X$, $\bfG'[X]$ is generated by the set $\calG'[X]$ of graphs $G$ such that $V(G)$ is a partition of $X$.
For example, 
\begin{align*}
\calG'[\{a,b\}]&=\{\tddeux{$a$}{$b$},\tdun{$a$}\tdun{$b$},\tdun{$a,b$}\hspace{3mm}\},\\
\calG'[\{a,b,c\}]&=\left\{\begin{array}{c}
\tdtroisdeux{$a$}{$b$}{$c$},\tdtroisdeux{$b$}{$a$}{$c$},\tdtroisdeux{$a$}{$c$}{$b$},
\gdtroisun{$a$}{$c$}{$b$},\tddeux{$a$}{$b$}\tdun{$c$},\tddeux{$a$}{$c$}\tdun{$b$},
\tddeux{$b$}{$c$}\tdun{$a$},\tdun{$a$}\tdun{$b$}\tdun{$c$},\\
\tddeux{$a,b$}{$c$}\hspace{2mm},\tddeux{$a,c$}{$b$}\hspace{2mm},\tddeux{$b,c$}{$a$}\hspace{2mm},
\tdun{$a,b$}\hspace{3mm}\tdun{$c$},\tdun{$a,c$}\hspace{3mm}\tdun{$b$},\tdun{$b,c$}\hspace{3mm}\tdun{$a$},
\tdun{$a,b,c$}\hspace{3mm}
\end{array}\right\}.
\end{align*}
The product is given by the disjoint union of these graphs. If $G\in \calG'[A\sqcup B]$, then
\[\Delta_{A,B}(G)=\begin{cases}
G_{\mid A}\otimes G_{\mid B}\mbox{ if $A$ and $B$ are union of vertices of $G$},\\
0\mbox{ otherwise}.
\end{cases}\]
For example, in $\bfG'[\{a,b,c\}]$,
\begin{align*}
\Delta_{\{a,b\},\{c\}}(\tddeux{$a,b$}{$c$}\hspace{2mm})
&=\tdun{$\{a,b\}$}\hspace{4mm}\otimes \tdun{$c$},&
\Delta_{\{c\},\{a,b\}}(\tddeux{$a,b$}{$c$}\hspace{2mm})
&=\tdun{$c$}\otimes\tdun{$\{a,b\}$}\hspace{4.5mm} ,&
\Delta_{\{a,c\},\{b\}}(\tddeux{$a,b$}{$c$}\hspace{2mm})&=0.
\end{align*}

For any graph $G\in \calG'[X]$,
\[\delta(G)=\sum_{\sim\in \eq_c[G]} G/\sim\otimes G\mid \sim,\]
noticing that the vertices of $G/\sim$ are here disjoint unions of vertices of $G$, so form a partition of $X$. 
For example, in $\bfG'[\{a,b,c\}]$,
\begin{align*}
\delta(\tddeux{$a,b$}{$c$}\hspace{2mm})&=\tddeux{$a,b$}{$c$}\hspace{2mm}\otimes 
\tdun{$a,b$}\hspace{3mm}\tdun{$c$}
+\tdun{$a,b,c$}\hspace{5mm}\otimes \tddeux{$a,b$}{$c$}\hspace{2mm},\\
\delta(\tdtroisdeux{$a$}{$b$}{$c$})&=\tdtroisdeux{$a$}{$b$}{$c$}\otimes \tdun{$a$}\tdun{$b$}\tdun{$c$}
+\tdun{$a,b,c$}\hspace{5mm}\otimes \tdtroisdeux{$a$}{$b$}{$c$}
+\tddeux{$a,b$}{$c$}\hspace{2mm} \otimes \tddeux{$a$}{$b$}\tdun{$c$}
+\tddeux{$a$}{$b,c$}\hspace{2mm} \otimes \tddeux{$b$}{$c$}\tdun{$a$},\\
\delta(\gdtroisun{$a$}{$c$}{$b$})&=\gdtroisun{$a$}{$c$}{$b$}\otimes \tdun{$a$}\tdun{$b$}\tdun{$c$}
+\tdun{$a,b,c$}\hspace{5mm}\otimes  \gdtroisun{$a$}{$c$}{$b$}
+\tddeux{$a,b$}{$c$}\hspace{2mm} \otimes \tddeux{$a$}{$b$}\tdun{$c$}
+\tddeux{$a,c$}{$b$}\hspace{2mm} \otimes \tddeux{$b$}{$c$}\tdun{$b$}
+\tddeux{$b,c$}{$a$}\hspace{2mm} \otimes \tddeux{$b$}{$c$}\tdun{$a$}.
\end{align*}\end{example}

\section{Bosonic Fock functors}

\subsection{The classical bosonic Fock functor}

Let us recall the definition of the bosonic Fock functor $\calF$ \cite{Aguiar2010}. 
\begin{itemize}
\item If $\bfP$ is a species, then $\calF[\bfP]$ is the graded vector space
\[\calF[\bfP]=\bigoplus_{n=0}^\infty \coinv(\bfP[n])
=\bigoplus_{n=0}^\infty \frac{\bfP[n]}{\vect(\bfP[\sigma](p)-p\mid \sigma\in \sym_n,\: p\in \bfP[n])}.\]
\item If $f:\bfP\longrightarrow \bfQ$ is a species morphism, then
\[\calF[f]=\bigoplus_{n=0}^\infty \overline{f[n]},\]
where $\overline{f[n]}:\coinv(\bfP[n])\longrightarrow \coinv(\bfQ[n])$
is the linear map induced by $f[n]$ (which is compatible with the action of $\sym_n$, as a species morphism). 
\end{itemize}

It is proved in \cite{Aguiar2010} that if $\bfP$ is a twisted algebra (resp. coalgebra, bialgebra), then $\calF[\bfP]$ is a graded algebra
(resp. coalgebra, bialgebra), with the product and coproduct given by
\begin{align*}
&\forall p\in \bfP[k],\: \forall q\in \bfP[l],&
m(\overline{p}\otimes \overline{q})&=\overline{\bfP[\sigma_{k,l}]\circ m_{[k],[l]}(p\otimes q)},\\
&\forall p\in \bfP[k+l],&
\Delta(\overline{p})&=\sum_{A\sqcup B=[n]} \overline{(\bfP[\sigma_A]\otimes \bfP[\sigma_B])\circ \Delta_{A,B}(p)},
\end{align*}
where for any $I\subseteq [n]$, $\sigma_I:I\longrightarrow [|I|]$ is the unique increasing bijection
and for any $k,l\geq 0$,
\[\sigma_{k,l}:\left\{\begin{array}{rcl}
[k]\sqcup[l]&\longrightarrow&[k+l]\\
i\in [k]&\longrightarrow&i,\\
j\in [l]&\longrightarrow&k+j.
\end{array}\right.\]

The following theorem is proved in \cite[Theorem 61]{Foissy39}:

\begin{theo}
Let $(\bfP,m,\Delta,\delta)$ be a double twisted bialgebra. Then $\calF[\bfP]$ is a double bialgebra, with the coproduct $\delta$ 
given on any $p\in \bfP[n]$ by
\[\delta(\overline{p})=(\pi_n\otimes \pi_n)\circ \delta[n](p),\]
where $\pi_n:\bfP[n]\longrightarrow\coinv(\bfP[n])$ is the canonical surjection. 
\end{theo}

\subsection{Coloured Fock functors}

If $\bfP$ and $\bfQ$ are twisted algebras (resp. coalgebras, bialgebras, double bialgebras), then 
their Hadamard tensor product $\bfP\boxtimes \bfQ$
is also a twisted algebra (resp. coalgebra, bialgebra, double bialgebra). We here shall take $\bfP=\bfT_V$:

\begin{defi}
Let $V$ be a vector space. The $V$-coloured Fock functor $\calF_V$ associates to any species $\bfP$ the graded vector space
$\calF_V[\bfP]=\calF[\bfT_V\boxtimes \bfP]$. 
If $f:\bfP\longrightarrow \bfQ$ is a species morphism, then $\calF_V[f]=\calF[\id_{\bfT_V}\otimes f]$.
\end{defi}

In other words, for any species $V$,
\begin{align*}
\calF_V[\bfP]&=\bigoplus_{n=0}^\infty \coinv(V^{\otimes n}\otimes \bfP[n])\\
&=\bigoplus_{n=0}^\infty \frac{V^{\otimes n}\otimes \bfP[n]}
{\vect(v_1\ldots v_n \otimes \bfP[\sigma](p)-v_{\sigma(1)}\ldots v_{\sigma(n)}\otimes p\mid
\sigma\in \sym_n,\: p\in \bfP[n],\:v_1,\ldots,v_n\in V)}\\
&=V^{\otimes n}\otimes_{\sym_n} \bfP[n],
\end{align*}
where $\sym_n$ acts on the left on $\bfP[n]$ by the action given by the species structure,
and acts on the right of $V^{\otimes n}$ by permutation of the tensors:
\begin{align*}
&\forall \sigma \in \sym_n,\: \forall v_1,\ldots,v_n \in V,&
v_1\ldots v_n \cdot \sigma&=v_{\sigma(1)}\ldots v_{\sigma(n)}.
\end{align*}
If $f:\bfP\longrightarrow \bfQ$ is a species morphism, then for any $p\in \bfP[n]$, for any $v_1,\ldots,v_n\in V$,
\[\calF_V[f](\overline{v_1\ldots v_n\otimes p})=\overline{v_1\ldots v_n\otimes f(p)}.\]

As $\bfT_V$ is a twisted bialgebra, and even a double twisted bialgebra if $V$ is a coalgebra:

\begin{prop} 
\begin{enumerate}
\item Let $V$ be a vector space. If $\bfP$ is a twisted algebra (resp. coalgebra, bialgebra),
then $\calF_V[\bfP]$ is a graded algebra (resp. coalgebra, bialgebra).
\item Let $(V,\delta_V)$ be a coalgebra. If $\bfP$ is a double twisted bialgebra, then $\calF_V[\bfP]$
is a double bialgebra. 
\end{enumerate}
\end{prop}

\begin{example}
Let us apply this to the double twisted bialgebra $\com$.
By construction,
\[\calF_V[\com]=S(V),\]
 with its usual product $m$  and coproduct $\Delta$,
defined on any $v\in V$ by $\Delta(v)=v\otimes 1+1\otimes v$. If moreover $(V,\delta_V)$ is a coalgebra, 
then $S(V)$ is a double bialgebra, with the coproduct $\delta$ obtained by multiplicative extension of $\delta_V$.
In particular, if $V=\K$ with its usual coalgebra structure, we obtain the double bialgebra $\K[X]$, 
with its coproducts defined by
\begin{align*}
\Delta(X)&=X\otimes 1+1\otimes X,&\delta(X)&=X\otimes X. 
\end{align*}
\end{example}

\subsection{The example of $\comp$}

Let us apply this to the double twisted bialgebra $\comp$. For any vector space $V$,
\[\calF_V[\comp]=T(S^+(V)).\]
In order to distinguish the different products, we  denote the concatenation product of 
$T(S^+(V))$ by $\mid$. Let us describe the quasishuffle product $\squplus$ of $\calF_V[\comp]$.
For any $p_1,\ldots,p_{k+l}\in S^+(V)$,
\[(p_1\mid \ldots \mid p_k) \squplus (p_{k+1}\mid \ldots \mid p_{k+l})
=\sum_{\sigma \in \QSh(k,l)} \left(\prod_{i\in\sigma^{-1}(1)} p_i\right)\mid \ldots \mid 
\left(\prod_{i\in\sigma^{-1}(\max(\sigma))} p_i\right),\]
where $\QSh(k,l)$ is the set of $(k,l$)-quasishuffles, that is to say surjections $\sigma:[k+l]\longrightarrow [\max(\sigma)]$
such that $\sigma(1)<\ldots <\sigma(k)$ and $\sigma(k+1)<\ldots <\sigma(k+l)$. 
The coproduct is given by deconcatenation: for any $p_1,\ldots,p_n\in S^+(V)$,
\[\Delta(p_1\mid \ldots \mid p_n)=\sum_{k=0}^n
p_1\mid \ldots \mid p_k\otimes p_{k+1}\mid \ldots \mid p_n.\]
If $(V,\delta_V)$ is a coalgebra, then $T(S^+(V))$ inherits a second coproduct $\delta$
making it a double bialgebra. Firstly, the coproduct of $V$ is extended in a coproduct $\delta_{S^+(V)}$ on $S^+(V)$
by multiplicativity. Then, for any $p_1,\ldots,p_k\in S^+(V)$, with Sweeder's notation $\delta_{S^+(V)}(p)=p'\otimes p''$,
\[\delta(p_1\mid \ldots \mid p_n)=\sum_{1\leq i_1<\ldots<i_p<k}
\left(\prod_{i=1}^{i_1}p'_i\right)\mid\ldots \mid \left(\prod_{i=i_p+1}^k p'_i\right)
\otimes (p''_1\mid \ldots\mid p''_{i_1})\squplus \ldots \squplus (p''_{i_p+1}\mid \ldots \mid p''_k).\]

When $(V,\cdot,\delta_V)$ is a commutative, non necessarily unitary bialgebra, 
using the canonical algebra epimorphism from $S^+(V)$ to $V$:

\begin{prop}\label{propTV}
\begin{enumerate}
\item Let $(V,\cdot)$ be a commutative, non necessarily unitary algebra. 
Then $T(V)$ is a bialgebra, with the following product $\squplus$ and coproduct $\Delta$:
\begin{itemize}
\item For any $v_1,\ldots,v_{k+l}\in V$,
\[v_1\ldots v_k \squplus v_{k+1} \ldots v_{k+l}
=\sum_{\sigma \in \QSh(k,l)} \left(\prod^\cdot_{i\in\sigma^{-1}(1)} v_i\right)\ldots
\left(\prod^\cdot_{i\in\sigma^{-1}(\max(\sigma))} v_i\right),\]
where the symbol $\displaystyle \prod^\cdot$ means that the product which is used is the product $\cdot$ of $V$.
\item For any $v_1,\ldots,v_n\in V$,
\begin{align*}
\Delta(v_1\ldots v_n)&=\sum_{k=0}^n v_1\ldots v_k\otimes v_{k+1}\ldots v_n,\\
\varepsilon_\Delta(v_1\ldots v_n)&=0\mbox{ if }n\geqslant 1.
\end{align*}
\end{itemize}
\item Let $(V,\cdot,\delta_V)$ be a commutative, non necessarily unitary bialgebra. 
Then $T(V)$ inherits a second coproduct, making it a double bialgebra: for any $v_1,\ldots,v_k\in V$,
\begin{align*}
&\delta(v_1\ldots v_k)\\
&=\sum_{1\leq i_1<\ldots<i_p<k}
\left(\prod_{1\leq i\leq i_1}^\cdot v'_i\right)\ldots  \left(\prod_{i_p+1\leq i\leq k}^\cdot  v'_i\right)
\otimes (v''_1 \ldots v''_{i_1})\squplus \ldots \squplus (v''_{i_p+1} \ldots  v''_k).
\end{align*}
The counit $\epsilon_\delta$ is given by
\[\epsilon_\delta(v_1\ldots v_k)=\begin{cases}
0\mbox{ if }k\geqslant 2,\\
\epsilon_V(v_1)\mbox{ if }k=1.
\end{cases}\]
\end{enumerate}\end{prop}

\begin{proof}
As $V$ is commutative, there exists an algebra morphism $\pi:S^+(V)\longrightarrow V$, sending any $v\in V$
to itself. It is extended as an algebra morphism from $T(S^+(V))$ to $T(V)$, which we also denote by $\pi$. 
It is immediate that
\begin{align*}
m\circ (\pi\otimes \pi)&=\pi\circ m,&\Delta \circ \pi&=(\pi\otimes \pi)\circ \Delta,
\end{align*}
and, if $V$ is a bialgebra,
\[\delta \circ \pi=(\pi\otimes \pi)\circ \delta.\]
 As $\pi$ is surjective, we obtain that $(T(V),\squplus,\Delta)$ is a bialgebra and, if $V$ is a commutative bialgebra, 
that $(T(V),\squplus, \Delta,\delta)$ is a double bialgebra.
\end{proof}

\begin{remark} \begin{enumerate}
\item Another proof of this result can be found in \cite{Foissy40} and \cite{Ebrahimi-Fard2017-2}. 
\item This is a generalization of the construction of quasishuffle algebras of \cite{Hoffman2000,Hoffman2020}, 
which we recover when $V$
is the algebra of a monoid. \\
\end{enumerate}
\end{remark}

\begin{example}
\begin{enumerate}
\item Let $V$ be the bialgebra associated to the monoid $(\N_{>0},+)$. A basis of $T(V)$ is given by compositions,
that is to say finite sequences of positive integers, and we recover the double bialgebra $\QSym$ of quasi-symmetric functions
\cite{Aguiar2006-2,Gelfand1995,Hazewinkel2005,Malvenuto2011,Stanley1999}.
\item Let $V=\K$, with its usual product and coproduct. In order to avoid confusions, we denote by $x$ the unit $1$ of $\K$,
seen as an element of $T(\K)$. Then
\begin{align*}
\Delta(x)&=x\otimes 1+1\otimes x,&\delta(x)&=x\otimes x.
\end{align*}
Moreover, for any $k,l\in \N$,
\[x^k \squplus x^l=\sum_{\max(k,l)\leq n\leq k+l} \frac{n!}{(n-k)!(n-l)!(k+l-n)!}x^n.\]
An easy induction proves that for any $n\geq 1$,
\[x^{\squplus n}=n!x^n+\mbox{terms $a_kx^k$ with $k<n$}.\]
We deduce that when the characteristic of $\K$ is zero, $x$ generates $(T(\K),\squplus)$.
In this case, $(T(\K),\squplus)$ is isomorphic to $\K[X]$, with its usual product and its coproducts defined by 
\begin{align*}
\Delta(X)&=X\otimes 1+1\otimes X,&\delta(X)&=X\otimes X. 
\end{align*}
In other words, identifying $\K[X]^{\otimes 2}$ and $\K[X,Y]$ with the algebra map
\[\left\{\begin{array}{rcl}
\K[X]\otimes \K[X]&\longrightarrow&\K[X,Y]\\
P(X)\otimes Q(X)&\longrightarrow&P(X)Q(Y),
\end{array}\right.\]
for any $P\in \K[X]$,
\begin{align*}
\Delta(P)(X,Y)&=P(X+Y),&
\delta(P)(X,Y)&=P(XY),\\
\varepsilon_\Delta(P)&=P(0),&
\epsilon_\delta(P)&=P(1). 
\end{align*}
\end{enumerate}\end{example}

\subsection{Coloured Fock functor and contractions}

Let $(\bfP,\delta)$ be a species with a contraction-extraction coproduct $\delta$. We put $\bfP'=\bfP\circ \com$.
Let us fix  a commutative, non necessarily unitary algebra  $(V,\cdot)$.
We define a map $\pi: \calF_V[\bfP']\longrightarrow \calF_V[\bfP]$ as follows:
if $v_1,\ldots,v_n \in V$ and $p\in \bfP[[n]/\sim]$, where $\sim \in \eq[n]$, 
let us choose a bijection $\sigma:[n]/\sim\longrightarrow [k]$. We put
\[\pi(\overline{v_1\ldots v_n \otimes p})=\overline{\left(\prod_{i\in \sigma^{-1}(1)}^\cdot v_i\right)
\ldots \left(\prod_{i\in \sigma^{-1}(k)}^\cdot v_i\right)\otimes \bfP[\sigma](p)} \in 
\coinv(V^{\otimes k}\otimes \bfP[k]),\]
where again the symbols $\displaystyle \prod^\cdot$ mean that these products are taken in $(V,\cdot)$. 

\begin{lemma}
The map $\pi$ is well-defined.
\end{lemma}

\begin{proof}
Let us first prove that this does not depend on the choice of the bijection $\sigma$. Let $\sigma':[n]/\sim\longrightarrow [k]$
be another bijection. We put $\tau=\sigma'\circ \sigma^{-1}\in \sym_k$, so that $\sigma'=\tau\circ \sigma$. 
\begin{align*}
&\overline{\left(\prod_{i\in \sigma'^{-1}(1)}^\cdot v_i\right)
\ldots \left(\prod_{i\in \sigma'^{-1}(k)}^\cdot v_i\right)\otimes \bfP[\sigma'](p)} \\
&=\overline{\left(\prod_{i\in \sigma^{-1}(\tau^{-1}(1))}^\cdot v_i\right)
\ldots \left(\prod_{i\in \sigma^{-1}(\tau^{-1}(k))}^\cdot v_i\right)\otimes\bfP[\tau]\circ\bfP[\sigma](p)} \\
&=\overline{\left(\prod_{i\in \sigma^{-1}(1)}^\cdot v_i\right)
\ldots \left(\prod_{i\in \sigma^{-1}(k)}^\cdot v_i\right)\otimes \bfP[\sigma](p)},
\end{align*}
by definition of $V^{\otimes k}\otimes_{\sym_k} \bfP[k]$. Let us now take $v_1,\ldots,v_n \in V$,
$\sim\in \eq[n]$ and $\tau\in \sym_n$. Let us choose a bijection $\sigma:X/\sim_\tau\longrightarrow [k]$.
Then $\sigma'=\sigma \circ \tau/\sim:X/\sim\longrightarrow [k]$ is a bijection. 
\begin{align*}
&\overline{\left(\prod_{i\in \sigma^{-1}(1)}^\cdot v_i\right)
\ldots \left(\prod_{i\in \sigma^{-1}(k)}^\cdot v_i\right)\otimes \bfP[\sigma]\circ \bfP'[\tau](p)} \\ 
&=\overline{\left(\prod_{i\in \sigma^{-1}(1)}^\cdot v_i\right)
\ldots \left(\prod_{i\in \sigma^{-1}(k)}^\cdot v_i\right)\otimes \bfP[\sigma\circ \tau/\sim](p)} \\ 
&=\overline{\left(\prod_{i\in \sigma^{-1}(1)}^\cdot v_i\right)
\ldots \left(\prod_{i\in \sigma^{-1}(k)}^\cdot v_i\right)\otimes \bfP[\sigma'](p)} \\ 
&=\overline{\left(\prod_{\tau^{-1}(i)\in \sigma'^{-1}(1)}^\cdot v_i\right)
\ldots \left(\prod_{\tau^{-1}(i)\in \sigma'^{-1}(k)}^\cdot v_i\right)\otimes \bfP[\sigma'](p)} \\ 
&=\overline{\left(\prod_{i\in \sigma'^{-1}(1)}^\cdot v_{\tau(i)}\right)
\ldots \left(\prod_{i\in \sigma'^{-1}(k)}^\cdot v_{\tau(i)}\right)\otimes \bfP[\sigma'](p)}.
\end{align*}
Note that the last equality holds because of the commutativity of $\cdot$, the factors in these products
being permuted according to $\tau$. \end{proof}

\begin{prop} Let $(V, \cdot)$ be a commutative, non necessarily unitary algebra.
\begin{enumerate}
\item Let us assume that $(\bfP,m)$ is a twisted algebra. Then $\pi:(\calF_V[\bfP'],m')\longrightarrow (\calF_V[\bfP],m)$
is an algebra morphism.
\item Let us assume that $(\bfP,\Delta)$ is a twisted coalgebra. Then $\pi:(\calF_V[\bfP'],\Delta')\longrightarrow 
(\calF_V[\bfP],\Delta)$ is a coalgebra morphism.
\end{enumerate}
\end{prop}

\begin{proof}
1. Let $k,l\geq 0$, $\sim_X\in \eq[k]$, $\sim_Y\in \eq[l]$, $p\in \bfP[[k]/\sim_X]$, $q\in \bfP[[l]/\sim_Y]$ 
and $v_1,\ldots,v_{k+l}\in V$. We choose bijections $\sigma_X:[k]/\sim_X\longrightarrow [k']$
and $\sigma_Y:[l]/\sim_Y\longrightarrow [l']$. Putting $\sim=\sim_X\sqcup \sim_Y$,
then $\sigma=\sigma_{k',l'}\circ (\sigma_X\otimes \sigma_Y):[k]\sqcup [l]/\sim\longrightarrow [k'+l']$ is a bijection.
\begin{align*}
&\pi(\overline{v_1\ldots v_k\otimes p}. \overline{v_{k+1}\ldots v_{k+l}\otimes q})\\
&=\pi(\overline{v_1\ldots v_{k+l}\otimes \bfP[\sigma]\circ m_{[k]/\sim_X,[l]/\sim_Y}(p\otimes q)})\\
&=\overline{\left(\prod^\cdot_{i\in \sigma_X^{-1}(1)}v_i\right)\ldots \left(\prod^\cdot_{i\in \sigma_X^{-1}(k')}v_i\right)
\left(\prod^\cdot_{i\in \sigma_Y^{-1}(1)}v_i\right)\ldots \left(\prod^\cdot_{i\in \sigma_Y^{-1}(l')}v_i\right)}\\
&\hspace{2cm}\overline{\otimes\bfP[\sigma_{k',l'}]\circ m_{[k'],[l']}(\bfP[\sigma_X](p)\otimes \bfP[\sigma_Y](q))}\\
&=\overline{\left(\prod^\cdot_{i\in \sigma_X^{-1}(1)}v_i\right)\ldots \left(\prod^\cdot_{i\in \sigma_X^{-1}(k')}v_i\right)
\otimes \bfP[\sigma_X](p)}
.\overline{\left(\prod^\cdot_{i\in \sigma_Y^{-1}(1)}v_i\right)\ldots \left(\prod^\cdot_{i\in \sigma_Y^{-1}(l')}v_i\right)
\otimes \bfP[\sigma_Y](q)}\\
&=\pi(\overline{v_1\ldots v_k\otimes p})\cdot \pi(\overline{v_{k+1}\ldots v_{k+l}\otimes q}).
\end{align*}

2. Let $n\geqslant 0$, $\sim \in \eq[n]$, $v_1,\ldots,v_n\in V$ and $p\in \bfP[[n]/\sim]$. 
Let $\sigma:[n]/\sim\longrightarrow [k]$ be a bijection. For any $I\subseteq [k]$, $\sigma_I\circ \sigma_{\mid \sigma^{-1}(I)}$
is a bijection denoted by $\sigma'_I$ from $\sigma^{-1}(I)$ to $[|I|]$. Then
\begin{align*}
&(\pi\otimes \pi)\circ \Delta(\overline{v_1\ldots v_n\otimes p})\\
&=\sum_{I\sqcup J=[k]} (\pi\otimes \pi) \left(\overline{\prod_{\sigma(i)\in I} v_i \otimes \bfP[\sigma_I\circ 
\sigma_{ \sigma^{-1}(I)}](p^{(1)_{\sigma^{-1}(I)}})}\otimes \overline{\prod_{\sigma(j)\in J} v_j \otimes 
\bfP[\sigma_J\circ \sigma_{ \sigma^{-1}(J)}](p^{(2)_{\sigma^{-1}(J)}})}\right)\\
&=\sum_{I\sqcup J=[k]} \overline{\left(\prod_{i\in \sigma_I'^{-1}(1)}^\cdot v_i\right)
\ldots \left(\prod_{i\in \sigma_I'^{-1}(|I|)}^\cdot v_i\right)\bfP[\sigma_I]\left(\bfP[\sigma](p)^{(1)_I}\right)}\\
&\hspace{2cm}
\overline{\left(\prod_{j\in \sigma_J'^{-1}(1)}^\cdot v_j\right)
\ldots \left(\prod_{j\in \sigma_J'^{-1}(|J|)}^\cdot v_j\right)\bfP[\sigma_J]\left(\bfP[\sigma](p)^{(2)_J}\right)}\\
&=\Delta\circ \pi(\overline{v_1\ldots v_n\otimes p}).  \qedhere
\end{align*}
 \end{proof}

Therefore,  if $(\bfP,m,\Delta)$ is a twisted bialgebra, $\pi:(\calF_V[\bfP'],m',\Delta')\longrightarrow 
(\calF_V[\bfP],m,\Delta)$ is a bialgebra morphism.

\begin{prop}
Let $(V,\cdot,\delta_V)$ be a  commutative, non necessarily unitary bialgebra. 
Let $(\bfP,m,\Delta)$ be a twisted bialgebra, with a contraction-extraction coproduct $\delta'$ as in 
Proposition \ref{propbialgebra}. There exists a unique coproduct $\delta$ on $\calF_V[\bfP]$, making it a double bialgebra,
such that $\pi:\calF_V[\bfP']\longrightarrow \calF_V[\bfP]$ is a morphism of double bialgebras.
\end{prop}

\begin{proof}
The uniqueness of $\delta$ comes from the surjectivity of $\pi$.
We define $\delta:\calF_V[\bfP]\longrightarrow \calF_V[\bfP]$ as follows:
for any $v_1,\ldots,v_n \in V$, for any $p\in \bfP[n]$,
\[\delta(\overline{v_1\ldots v_n\otimes p})=(\pi\otimes \id)\circ \delta'(\overline{v_1\ldots v_n\otimes p}),\]
seeing $\bfP$ as a subspecies of $\bfP'$. Let us prove that $\delta\circ \pi=(\pi\otimes \pi)\circ \delta'$.
Let $v_1\ldots v_n \in V$ and $p\in \bfP[[n]/\sim']$. 
Let us choose a bijection $\sigma':[n]/\sim'\longrightarrow [k']$. 
If $\sim\leq \sim'$, let us choose a bijection $\sigma_\sim:[n]/\sim\longrightarrow [k_\sim]$,
of the form $\sigma_\sim=\tau_\sim \circ \sigma'/\sim$, with $\tau$ a bijection from $[k']/\sim_{\sigma'}$ to $[k_\sim]$.
Then
\begin{align*}
\delta \circ \pi(\overline{v_1\ldots v_n\otimes p})&=\delta\left(\overline{
\left(\prod_{i\in \sigma'^{-1}(1)}^\cdot v_i\right)\ldots \left(\prod_{i\in \sigma'^{-1}(k')}^\cdot v_i\right)
\otimes \bfP'[\sigma'](p)}\right)\\
&=\sum_{\sim\leq \sim'}\pi\left(\overline{
\left(\prod_{i\in \sigma'^{-1}(1)}^\cdot v_i\right)'\ldots \left(\prod_{i\in \sigma'^{-1}(k')}^\cdot v_i\right)'
\otimes \bfP'[\sigma'](p^{(1)_\sim})}\right)\\
&\hspace{2cm} \otimes \overline{
\left(\prod_{i\in \sigma'^{-1}(1)}^\cdot v_i\right)''\ldots \left(\prod_{i\in \sigma'^{-1}(k')}^\cdot v_i\right)''
\otimes \bfP'[\sigma'](p^{(2)_\sim})}\\
&=\sum_{\sim\leq \sim'}\pi\left(\overline{
\left(\prod_{i\in \sigma'^{-1}(1)}^\cdot v'_i\right)\ldots \left(\prod_{i\in \sigma'^{-1}(k')}^\cdot v'_i\right)
\otimes \bfP[\sigma'/\sim](p^{(1)_\sim})}\right)\\
&\hspace{2cm} \otimes \overline{
\left(\prod_{i\in \sigma'^{-1}(1)}^\cdot v''_i\right)\ldots \left(\prod_{i\in \sigma'^{-1}(k')}^\cdot v''_i\right)
\otimes \bfP[\sigma'](p^{(2)_\sim})}\\
&=\sum_{\sim\leq \sim'}\overline{
\left(\prod_{i\in \sigma_\sim^{-1}(1)}^\cdot v'_i\right)\ldots \left(\prod_{i\in \sigma_\sim^{-1}(k_\sim)}^\cdot v'_i\right)
\otimes \bfP[\sigma_\sim](p^{(1)_\sim})}\\
&\hspace{2cm} \otimes \overline{
\left(\prod_{i\in \sigma'^{-1}(1)}^\cdot v''_i\right)\ldots \left(\prod_{i\in \sigma'^{-1}(k')}^\cdot v''_i\right)
\otimes \bfP[\sigma'](p^{(2)_\sim})}\\
&=\sum_{\sim\leq \sim'}\pi\left(\overline{v'_1\ldots v'_n \otimes p^{(1)_\sim}}\right)
\otimes \pi\left(\overline{v''_1\ldots v''_n \otimes p^{(2)_\sim}}\right)\\
&=(\pi\otimes \pi)\circ \delta'(\overline{v_1\ldots v_n\otimes p}),
\end{align*}
with Sweedler's notation $\delta_V(v)=v'\otimes v''$ for any $v\in V$ and $\delta_\sim(p)=p^{(1)_\sim}\otimes p^{(2)_\sim}$
for $p\in \bfP[[n]/\sim]$.
So, indeed, $(\pi\otimes \pi)\circ \delta'=\delta\circ \pi$. As $(\bfP',m',\Delta',\delta')$ is a double twisted bialgebra,
its quotient $(\bfP,m,\Delta,\delta)$ is a double twisted bialgebra. \end{proof}

\begin{cor} \label{corbifoncteur}
For any twisted bialgebra $\bfP$ with a compatible contraction-extraction coproduct, 
$\calF_\bullet[\bfP]$ is a functor from the category of commutative, non necessarily unitary bialgebras  
to the category of double bialgebras. 
\end{cor}

\subsection{Double bialgebras over $V$ from coloured Fock functors}

We introduced in \cite{Foissy42} the notion of double bialgebra over $V$. Let us now prove that coloured Fock functors
give such objects. 

\begin{prop}
Let $(V,\cdot,\delta_V)$ be a (not necessarily unitary) commutative and cocommutative bialgebra and let 
$(\bfP,m,\Delta)$ be a twisted bialgebra, with a contraction-extraction coproduct as in Proposition \ref{propbialgebra}.
The double bialgebra $(\calF_V[\bfP],m,\Delta,\delta)$ is a double bialgebra over $V$ in the sense of \cite{Foissy42},
with the coaction defined by
\begin{align*}
&\forall v_1,\ldots,v_k\in V,\:\forall p\in \bfP[k],&\rho(\overline{v_1\ldots v_k\otimes p})
&=\overline{v_1'\ldots v_k'\otimes p}\otimes \prod_{i\in [k]}^\cdot v''_i,
\end{align*}
with Sweedler's notation $\delta_V(v)=v'\otimes v''$ for any $v\in V$ and where the symbol
$\displaystyle \prod^\cdot$ means that the corresponding product is taken in $V$.
\end{prop}

\begin{proof}
As $V$ is commutative, $\rho$ is well-defined. The map $\rho$ is clearly a coaction of $V$ on $\calF_V[\bfP]$. 
Let us prove that the product is a $V$-comodule morphism. Let $v_1,\ldots,v_{k+l}\in V$, $p\in \bfP[k]$ and $q\in \bfP[l]$. 
\begin{align*}
&\rho\circ m(\overline{v_1\ldots v_k\otimes p}\otimes \overline{v_{k+1}\ldots v_{k+l}\otimes q})\\
&=\overline{v'_1\ldots v'_{k+l}\otimes \bfP[\sigma_{k,l}]\circ m_{[k],[l]}(p\otimes q)}\otimes 
v''_1\cdot \ldots \cdot v''_{k+l}\\
&= m(\overline{v'_1\ldots v'_k\otimes p}\otimes \overline{v'_{k+1}\ldots v'_{k+l}\otimes q})\otimes 
(v''_1\cdot \ldots \cdot v''_k)\cdot (v''_{k+1}\cdot \ldots \cdot  v''_{k+l})\\
&=m_{13,24}\circ (\rho\otimes \rho)(\overline{v_1\ldots v_k\otimes p}\otimes \overline{v_{k+1}\ldots v_{k+l}\otimes q}).
\end{align*}

Let us now prove that the coproduct $\Delta$ is a $V$-comodule morphism. Let $v_1,\ldots,v_k\in V$ and $p\in \bfP[k]$. 
\begin{align*}
(\Delta \otimes \id)\circ \rho(\overline{v_1\ldots v_k\otimes p})
&=\sum_{I\sqcup J=[k]}\left(\overline{\prod_{i\in I}^\cdot v'_i\otimes p^{(1)_I}}\right)\otimes 
\left(\overline{\prod_{i\in J}^\cdot v'_i\otimes p^{(1)_I}}\right)\otimes v''_1\cdot \ldots \cdot v''_n\\
&=\sum_{I\sqcup J=[k]}\left(\overline{\prod_{i\in I}^\cdot v'_i\otimes p^{(1)_I}}\right)\otimes 
\left(\overline{\prod_{i\in J}^\cdot v'_i\otimes p^{(1)_I}}\right)\otimes
\left(\prod_{i\in I}^\cdot v_i''\right)\cdot \left(\prod_{i\in J}^\cdot v_i''\right)\\
&=m_{1,3,24}\circ (\rho \otimes \rho)\circ \Delta(\overline{v_1\ldots v_k\otimes p}).
\end{align*}
We used the commutativity of $V$ for the second equality. 

Let us finally prove the compatibility between $\rho$ and $\delta$. Let $v_1,\ldots,v_k\in V$ and $p\in \bfP[k]$. 
\begin{align*}
(\delta\otimes \id)\circ \rho(\overline{v_1\ldots v_k\otimes p})&=
\sum_{\sim \in \eq[k]} \overline{\prod_{C\in [k]/\sim} \left(\prod_{i\in C}^\cdot v'_i\right)\otimes p^{(1)_\sim}}
\otimes \overline{v''_1\ldots v''_k \otimes p^{(2)_\sim}}\otimes v'''_1\cdot \ldots \cdot v'''_k\\
&=\sum_{\sim \in \eq[k]} \overline{\prod_{C\in [k]/\sim} \left(\prod_{i\in C}^\cdot v'_i\right)\otimes p^{(1)_\sim}}
\otimes \overline{v'''_1\ldots v'''_k \otimes p^{(2)_\sim}}\otimes 
\prod_{C\in [k]/\sim}^\cdot  \prod_{i\in C}^\cdot v''_i\\
&=(\id \otimes c)\circ (\rho \otimes \id)\circ \delta(\overline{v_1\ldots v_k\otimes p}).
\end{align*}
We used the commutativity and cocommutativity of $V$ for the second equality.
\end{proof}

\subsection{Double bialgebras of graphs}

Let $(V,\cdot,\delta_V)$ be a non necessarily unitary, commutative bialgebra. The double bialgebra $\calF_V[\bfG]$
is generated by graphs $G$ which any vertex $v$ is decorated by an element $d_G(v)$, with condititions of linearity in each vertex.
For example, if $v_1,v_2,v_3,v_4\in V$ and $\lambda_2,\lambda_4\in \K$,
if $w_1=v_1+\lambda_2v_2$ and $w_2=v_3+\lambda_4 v_4$,
\begin{align*}
\tddeux{$w_1$}{$w_2$}\:&=\tddeux{$v_1$}{$v_3$}\:+\lambda_4\tddeux{$v_1$}{$v_4$}\:
+\lambda_2\tddeux{$v_2$}{$v_3$}\:+\lambda_2\lambda_4\tddeux{$v_2$}{$v_4$}\:.
\end{align*}

The product is given by the disjoint union of graphs, the decorations being untouched. For any graph $G$,
\[\Delta(G)=\sum_{V(G)=A\sqcup B} G_{\mid A}\otimes G_{\mid B},\]
the decorations being untouched. Moreover,
\[\delta(G)=\sum_{\sim \in \eq_c[G]} G/\sim \otimes G\mid \sim.\]
Any vertex $w\in V(G/\sim)=V(G)/\sim$ is decorated by
\[\prod_{v\in w}^\cdot d_G(v)',\]
where the symbol $\displaystyle \prod^\cdot$ means that the product is taken in $V$ (recall that any vertex of $V(G/\sim)$
is a subset of $V(G)$). Any vertex $v\in V(G\mid \sim)=V(G)$ is decorated by $d_G(v)''$.
We use Sweedler's notation $\delta_V(v)=v'\otimes v''$, and it is implicit that in the expression of $\delta(G)$, 
everything is developed by multilinearity in the vertices. For example, if $v_1,v_2,v_3\in V$,
\begin{align*}
\Delta(\tddeux{$v_1$}{$v_2$}\:)&=\tddeux{$v_1$}{$v_2$}\:\otimes 1
+1\otimes \tddeux{$v_1$}{$v_2$}\:+\tdun{$v_1$}\:\otimes \tdun{$v_2$}\:+
\tdun{$v_2$}\:\otimes \tdun{$v_1$}\:,\\
\Delta(\tdtroisdeux{$v_1$}{$v_2$}{$v_3$}\:)&=\tdtroisdeux{$v_1$}{$v_2$}{$v_3$}\:\otimes 1
+1\otimes \tdtroisdeux{$v_1$}{$v_2$}{$v_3$}\:+\tddeux{$v_1$}{$v_2$}\:\otimes \tdun{$v_3$}\:
+\tddeux{$v_2$}{$v_3$}\:\otimes \tdun{$v_1$}\:+\tdun{$v_1$}\:\tdun{$v_3$}\:\otimes \tdun{$v_2$}\:\\
&+\tdun{$v_3$}\:\otimes  \tddeux{$v_1$}{$v_2$}\:+\tdun{$v_1$}\:\otimes  \tddeux{$v_2$}{$v_3$}\:
+\tdun{$v_2$}\:\otimes \tdun{$v_1$}\:\tdun{$v_3$}\:,\\
\Delta(\hspace{1mm}\gdtroisun{\hspace{-4mm}$v_1$}{$v_3$}{\hspace{-1mm}$v_2$}\:)&=
\hspace{1mm}\gdtroisun{\hspace{-4mm}$v_1$}{$v_3$}{\hspace{-1mm}$v_2$}\:\otimes 1
+1\otimes \hspace{1mm}\gdtroisun{\hspace{-4mm}$v_1$}{$v_3$}{\hspace{-1mm}$v_2$}\:
+\tddeux{$v_1$}{$v_2$}\:\otimes \tdun{$v_3$}\:
+\tddeux{$v_2$}{$v_3$}\:\otimes \tdun{$v_1$}\:+\tddeux{$v_1$}{$v_3$}\:\otimes \tdun{$v_2$}\:\\
&+\tdun{$v_3$}\:\otimes  \tddeux{$v_1$}{$v_2$}\:+\tdun{$v_1$}\:\otimes  \tddeux{$v_2$}{$v_3$}\:
+\tdun{$v_2$}\:\otimes \tddeux{$v_1$}{$v_3$}\:,\\
\\
\delta(\tddeux{$v_1$}{$v_2$}\:)&=\tddeux{$v_1'$}{$v_2'$}\:\otimes \tdun{$v_1''$}\hspace{2mm}\tdun{$v_2''$}\hspace{2mm}
+\tdun{$v_1'\cdot v_2'$}\hspace{6mm}\otimes \tddeux{$v_1''$}{$v_2''$}\hspace{2mm},\\
\delta(\tdtroisdeux{$v_1$}{$v_2$}{$v_3$}\:)&=
\tdtroisdeux{$v_1'$}{$v_2'$}{$v_3'$}\:\otimes \tdun{$v_1''$}\hspace{2mm}\tdun{$v_2''$}\hspace{2mm}
\tdun{$v_3''$}\hspace{2mm}+\tdun{$v_1'\cdot v_2'\cdot v_3'$}\hspace{10mm}\otimes 
\tdtroisdeux{$v_1''$}{$v_2''$}{$v_3''$}\:
+\tddeux{$v_1'\cdot v_2'$}{$v_3'$}\hspace{6mm} \otimes \tddeux{$v_1''$}{$v_2''$}\hspace{2mm}\tdun{$v_3''$}\hspace{2mm}
+\tddeux{$v_1'$}{$v_2'\cdot v_3'$}\hspace{6mm} \otimes \tddeux{$v_2''$}{$v_3''$}\hspace{2mm}\tdun{$v_1''$}\hspace{2mm}
,\\
\delta(\hspace{1mm}\gdtroisun{\hspace{-4mm}$v_1$}{$v_3$}{\hspace{-1mm}$v_2$}\:)&=
\hspace{1mm}\gdtroisun{\hspace{-4mm}$v_1'$}{$v_3'$}{\hspace{-1mm}$v_2'$}\:\otimes
\tdun{$v_1''$}\hspace{2mm}\tdun{$v_2''$}\hspace{2mm}
\tdun{$v_3''$}\hspace{2mm}+\tdun{$v_1'\cdot v_2'\cdot v_3'$}\hspace{10mm}\otimes \hspace{1mm}\gdtroisun{\hspace{-4mm}$v_1''$}{$v_3''$}{\hspace{-1mm}$v_2''$}\:\\
&+\tddeux{$v_1'\cdot v_2'$}{$v_3'$}\hspace{6mm} \otimes \tddeux{$v_1''$}{$v_2''$}\hspace{2mm}\tdun{$v_3''$}\hspace{2mm}
+\tddeux{$v_1'\cdot v_3'$}{$v_2'$}\hspace{6mm} \otimes \tddeux{$v_1''$}{$v_3''$}\hspace{2mm}\tdun{$v_2''$}\hspace{2mm}
+\tddeux{$v_2'\cdot v_3'$}{$v_1'$}\hspace{6mm} \otimes \tddeux{$v_2''$}{$v_3''$}\hspace{2mm}\tdun{$v_1''$}\hspace{2mm}.
\end{align*}
For any $V$-decorated graph,
\[\epsilon_\delta(G)=\begin{cases}
\displaystyle \prod_{v\in V(G)} \epsilon_V(d_G(v)) \mbox{ if }E(G)=\emptyset,\\
0\mbox{ otherwise}. 
\end{cases}\]

When one takes $V=\K$, with its usual bialgebraic structure, the multilinearity condition on the decorations implies that
a basis of $\calF_\K[\bfG]$ is given by graphs which all vertices are decorated by 1, which we identifty with
the set of graphs. We recover in this way the double bialgebra of graphs described in \cite{Foissy36}.
If $V$ is the bialgebra of the semigroup $(\N_{>0},+)$, we recover the double bialgebra of weighted graphs also described
in \cite{Foissy36}.

\bibliographystyle{amsplain}
\bibliography{biblio}

\end{document}